\documentclass[]{amsart}
\usepackage[latin1]{inputenc}
\usepackage{amsmath}
\usepackage{amscd}
\usepackage{amsfonts}
\usepackage{amssymb}
\usepackage{mathrsfs,cite,float}
\usepackage{graphicx}
\usepackage{tikz}
\usetikzlibrary{shapes,arrows,automata,positioning,decorations.pathreplacing,decorations.markings}

\usepackage[all,cmtip,2cell]{xy}
\UseTwocells
\xyoption{poly}

\def\softd{{\leavevmode\setbox1=\hbox{d}%
		\hbox to 1.05\wd1{d\kern-0.2ex{\char039}\hss}}}

\newtheorem{thm}{Theorem}[section]
\newtheorem{lem}[thm]{Lemma}
\newtheorem{pro}[thm]{Proposition}

\theoremstyle{definition}
\newtheorem{dfn}{Definition}[section]
\theoremstyle{remark}

\title[Inductive groupoids and cross-connections of regular semigroups]{Inductive groupoids and cross-connections\\ of regular semigroups}
\author{P. A. Azeef Muhammed}
\address{Institute of Natural Sciences and Mathematics,
	Ural Federal University,
	Lenina 51,
	620000 Ekaterinburg, Russia.}
\email{a.a.parail@urfu.ru, azeefp@gmail.com}
\author{M. V. Volkov}
\address{Institute of Natural Sciences and Mathematics,
	Ural Federal University,
	Lenina 51,
	620000 Ekaterinburg, Russia.}
\email{mikhail.volkov@usu.ru}
\thanks{The authors acknowledge the financial support by the Competitiveness Enhancement Program of the Ural Federal University. They were also supported by the Russian Foundation for Basic Research, grant no. 17-01-00551, and the Ministry of Education and Science of the Russian Federation, project no. 1.3253.2017/4.6}
\begin{document}
	\maketitle
	\begin{abstract}
		There are two major structure theorems for an arbitrary regular semigroup using categories, both due to Nambooripad. The first construction using inductive groupoids departs from the biordered set structure of a given regular semigroup. This approach belongs to the realm of the celebrated Ehresmann--Schein--Nambooripad Theorem and its subsequent generalisations. The second construction is a generalisation of Grillet's work on cross-connected partially ordered sets, arising from the principal ideals of the given semigroup. In this article, we establish a direct equivalence between these two seemingly different constructions. We show how the cross-connection representation of a regular semigroup may be constructed directly from the inductive groupoid of the semigroup, and vice versa. 
	\end{abstract}
	
	\section{Background and Overview}
	In 1970, Munn published a seminal article \cite{munn} describing the structure of a fundamental inverse semigroup from its semilattice of idempotents. He showed that every fundamental inverse semigroup can be realised as a certain semigroup of partial bijections of its semilattice of idempotents. Two distinct generalisations of this result to fundamental regular semigroups were established in the 1970s. The first approach initiated by Hall \cite{hall} and later refined by Grillet \cite{gril,gril1,gril2} was based on the observation that the ideal structure of a semigroup arises from a cross-connected pair of partially ordered sets. The second approach, closer to Munn's original one, relied on the idempotent structure of the semigroup, and came from a completely isolated source in India. Nambooripad, in his doctoral thesis \cite{kssthesis, ksssf} at the newly founded Department of Mathematics in the University of Kerala, identified (and axiomatised) the structure of the idempotents of a regular semigroup as a regular biordered set and constructed a fundamental regular semigroup as an exact generalisation of Munn's representation. 
	
	An equally promising approach to study arbitrary inverse semigroups (not only fundamental ones) was proposed by Schein \cite{schein, schein1} in 1965, connecting the ideas of Wagner's school of inverse semigroups with Ehresmann's work on ordered groupoids \cite{ehrs, ehrs1} related to pseudogroups. (See \cite[Notes on Chapter 4]{lawson} for a detailed history.)  Although initially unaware of these developments, Nambooripad \cite{kssthesis,ksssf1} gave a general construction of arbitrary regular semigroups and in addition, he placed it on a proper conceptual framework by describing a category equivalence between the category of inductive groupoids\footnote{Nambooripad had initially called these structures as \emph{regular systems} in his Ph.D. thesis. We use the term \emph{inductive groupoid} throughout the article as used by Nambooripad in his later Memoirs\cite{mem} in the context of regular semigroups. The inductive groupoids of inverse semigroups form a special case of Nambooripad's inductive groupoids.} and the category of regular semigroups. 
	
	Lawson \cite{lawson} in his book on inverse semigroups gathered the aforementioned results from diverse origins, in the context of inverse semigroups, and named the resulting statement the Ehresmann--Schein--Nambooripad (ESN) Theorem (also see \cite{hollings}). Nambooripad's results on regular semigroups were polished in the backdrop of Schein's work and published in 1979 as Memoirs of the American Mathematical Society \cite{mem}; this has remained as one of the most seminal articles on the structure theory of regular semigroups till the present date. It should also be mentioned here that at the same time, a very similar approach using structure mappings was independently initiated by Meakin \cite{stmp0,stmp,stmp1}. As we shall see later, this may be seen as a variant of Nambooripad's initial approach.
	
	In 1978, Nambooripad \cite{bicxn} explored the relationship between the two constructions of fundamental regular semigroups: his biordered set approach and Grillet's cross-connection approach, and showed that they are equivalent, thus completing the following triangle.\vspace{.5cm}\\ 
	
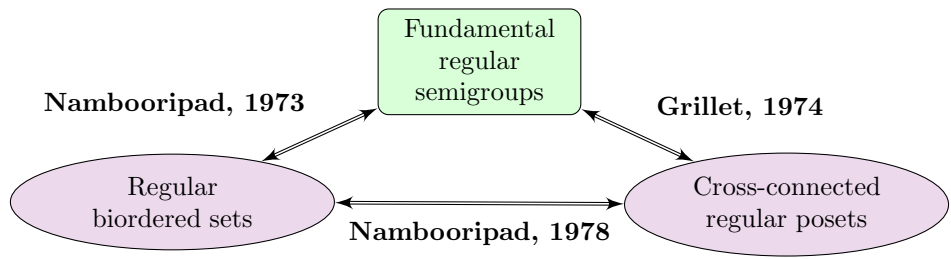
	\begin{tikzpicture}[align=center,node distance=2cm]
	\tikzstyle{block} = [rectangle, draw, fill=green!15, text width=7em, text centered, rounded corners, minimum height=4em]
	\tikzstyle{cloud} = [draw, ellipse,fill=violet!15, text width=8em, text centered,  node distance=3cm,
	minimum height=2em]
	
	\node [block] (rs) {Fundamental regular semigroups};
	\node [cloud, below right=0.7cm and 1.2cm of rs] (cxn) {Cross-connected regular posets};
	\node [cloud, below left=0.7cm and 1.2cm of rs] (ind) {Regular biordered sets};
	
	\draw[latex'-latex',double] (ind) -- node[label=270:{\textbf{Nambooripad, 1978}}] {} (cxn);
	\draw[latex'-latex',double] (rs) -- node[label=30:{\textbf{Grillet, 1974}}] {} (cxn);
	\draw[latex'-latex',double] (rs) -- node[label=120:{\textbf{Nambooripad, 1973}}] {} (ind);
	\end{tikzpicture}\vspace{.5cm}\\
	
	In 1989, Nambooripad \cite{cross0} (revised in 1994 as \cite{cross}) generalised Grillet's construction of fundamental regular semigroups to arbitrary regular semigroups by replacing partially ordered sets with suitable categories. This construction, however, has not attained the research interest it deserved, in deep contrast to \cite{mem}. One of the reasons for this is that both treatises \cite{cross0} and \cite{cross} appeared as local publications in Thiruvananthapuram, India, and were basically unavailable internationally. Another, more   conceptual reason for the cross-connection theory to stay in limbo for decades is that constructing arbitrary regular semigroups from cross-connected categories demanded extensive use of the notions and language of category theory at a quite advanced level. This made the entry threshold of the cross-connection machinery look relatively high, and besides that, the machinery itself constituted a ``technical tour de force'', quoting \cite{jmar}.    
	
	In \cite{cross}, Nambooripad had proved the equivalence of the category of regular semigroups with the category of cross-connections. So \cite{mem} combined with \cite{cross}, by transitivity, implies that Nambooripad's cross-connection theory is equivalent to its mainstream predecessor, inductive groupoid theory. In this article, we explore the direct relationship between cross-connections and inductive groupoids, in the realm of regular semigroups. Namely, we directly construct the inductive groupoid of a regular semigroup starting from the cross-connection representation of the semigroup, and vice versa. This completes the following triangle in the case of arbitrary regular semigroups.\vspace{.5cm}\\
	
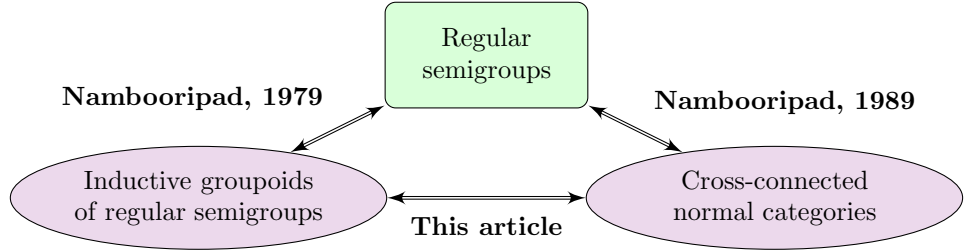
	\begin{tikzpicture}[align=center,node distance=2cm]
	\tikzstyle{block} = [rectangle, draw, fill=green!15, text width=7em, text centered, rounded corners, minimum height=4em]
	\tikzstyle{cloud} = [draw, ellipse,fill=violet!15, text width=9.4em, text centered,  node distance=3cm,
	minimum height=2em]
	
	\node [block] (rs) {Regular semigroups};
	\node [cloud, below right=0.7cm and .7cm of rs] (cxn) {Cross-connected normal categories};
	\node [cloud, below left=0.7cm and .7cm of rs] (ind) {Inductive groupoids of regular semigroups};
	
	\draw[latex'-latex',double] (ind) -- node[label=270:{\textbf{This article}}] {}  (cxn);
	\draw[latex'-latex',double] (rs) -- node[label=30:{\textbf{Nambooripad, 1989}}] {} (cxn);
	\draw[latex'-latex',double] (rs) -- node[label=120:{\textbf{Nambooripad, 1979}}] {} (ind);
	
	\end{tikzpicture}
	\vspace{.5cm}
	
	As the reader will see, the equivalence between inductive groupoids and cross-connections constructed in the present article is not a mere ``composition'' of the equivalencies established by Nambooripad. We believe that our results may shed some new light on both approaches since our study reveal some essential differences in the ways inductive groupoids and cross-connections encode regular semigroups.
	
	The article is divided into five sections. In Section \ref{pre}, we set up the notation and give the necessary preliminaries regarding semigroups, categories, biordered sets, inductive groupoids, and cross-connections needed for our discussion. In Section \ref{ind}, starting from the cross-connection representation of a regular semigroup $S$, we directly construct the inductive groupoid of the cross-connection and show that this groupoid is isomorphic to the inductive groupoid of $S$. In Section \ref{cxn}, conversely, given the inductive groupoid of the semigroup $S$, we construct a cross-connection such that it is cross-connection isomorphic to the cross-connection of $S$. In the final section, we discuss recent research developments related to \cite{mem}, namely the problems related to the maximal subgroups of the free idempotent generated semigroup on a biordered set and various non-regular generalisations of the ESN theorem. We conclude by suggesting some possible future directions on these problems via the theory of cross-connections. 
	
	In a follow-up article \cite{indcxn2} we extend the equivalence constructed in the present article to a category equivalence between abstract inductive groupoids and cross-connections. 
	
	\section{Preliminaries}\label{pre}
	Since this article aims to relate several seemingly different approaches to studying regular semigroups, we need to recall the main concepts involved in these approaches in some detail and to present these concepts using coherent notation. Therefore the list of the necessary preliminaries has become relatively lengthy even though we have not set the (non-realistic) goal of making the article fully self-contained. We refer the reader to Clifford and Preston \cite{clif} for standard notions from semigroup theory, to MacLane \cite{mac} or Higgins \cite{higginscat} for category theory, to Nambooripad's Memoirs \cite{mem} for details on biordered sets and inductive groupoids, to Grillet's book \cite{grillet} for structure mappings and Grillet's cross-connections, and to Nambooripad's treatise \cite{cross} for Nambooripad's cross-connection theory.  
	
	\subsection{Semigroups}
	An element $e$ of a semigroup $S$ is called an \emph{idempotent} if $e^2=e$. We denote by $E(S)$ the set of all idempotents of $S$. An element $b$ of a semigroup $S$ is called an \emph{inverse} of an element $a$ in $S$ if
	$$aba=a\text{ and }bab=b.$$
	We denote the set of all the inverses of an element $a$ in $S$ by $V(a)$ and an inverse of $a$ is denoted by $a'$. A semigroup $S$ is said to be \emph{regular} if every element in $S$ has at least one inverse element. A semigroup $S$ is said to be \emph{inverse} if every element in $S$ has a unique inverse.
	
	If $a$ is an element of a semigroup $S$, the \emph{principal left ideal} generated by $a$ is the subset $Sa \cup \{a\}$, and is denoted by $S^1a$ (or simply $Sa$, when $S$ is regular). Similarly the \emph{principal right ideal} $aS^1$ (or $aS$) is the subset $aS \cup \{a\}$. The \emph{Green} $\mathscr{L}$ relation on $S$ is defined by the rule that $a \mathscr{L} b$ if and only if $S^1a = S^1b$. Similarly the $\mathscr{R}$ relation is defined by $a \mathscr{R} b$ if and only if $aS^1 = bS^1$. We define the $\mathscr{H}$ relation as $\mathscr{L}\cap \mathscr{R}$ and the $\mathscr{D}$ relation as $\mathscr{L}\vee \mathscr{R}$ (in the lattice of all equivalences on $S$). It is well known that $\mathscr{D}=\mathscr{L} \circ \mathscr{R}= \mathscr{R}\circ\mathscr{L}$ where $\circ$ is the usual product of binary relations.
	
	An equivalence relation $\mathscr{C}$ on a semigroup $S$ is called a \emph{congruence} if $a\:\mathscr{C} \:b$, $c\:\mathscr{C}\:d$ implies $ac\:\mathscr{C}\:bd$. A semigroup $S$ is said to be \emph{fundamental} if the equality relation on $S$ is the only congruence contained in $\mathscr{H}$.
	
	\subsection{Categories}
	A \emph{category} $\mathcal{C}$ is a class of objects (denoted $v\mathcal{C}$) together with a collection of disjoint classes, denoted by $\mathcal{C}(a,b)$; one for each pair $(a,b)$ of objects in $v\mathcal{C}$. An element $f$ of $\mathcal{C}(a,b)$ is called a \emph{morphism} from $a$ to $b$; we often write $f\colon a\to b$ to say that $f$ is a morphism from $a$ to $b$. For each triple $(a,b,c)$ of objects in $\mathcal{C}$, a {composition} function $\mathcal{C}(a,b)\times\mathcal{C}(b,c)\to\mathcal{C}(a,c)$ is defined. Given morphisms $f\colon a \to b$ and $g\colon b \to c$, their composition will be written $fg$. Further, \emph{associativity} of composition and \emph{existence of identities} are assumed in a category. Associativity means that for each quadruple $(a,b,c,d)$ of objects in $\mathcal{C}$ and for each triple of morphisms $(f,g,h)\in\mathcal{C}(a,b)\times\mathcal{C}(b,c)\times\mathcal{C}(c,d)$, the expressions $(fg)h$ and $f(gh)$ represent the same morphism in $\mathcal{C}(a,d)$. Existence of identities means that for every object $c\in v\mathcal{C}$, there exists a morphism $1_c\colon c\to c$ such that $f1_c=f$ for every morphism $f\colon a\to c$ and $1_cg=g$ for every morphism $g\colon c\to b$. We shall often identify the identity morphism $1_c$ at an object $c\in v\mathcal{C}$ with the object $c$. With this convention, the morphisms of a category $\mathcal{C}$ completely determine $\mathcal{C}$, and having this in mind, we shall denote the class of all morphisms of $\mathcal{C}$ by $\mathcal{C}$ itself. 
	
	A category $\mathcal{C}$ is called \emph{small} if both $v\mathcal{C}$ and $\mathcal{C}$ are sets rather than proper classes. A category $\mathcal{C}$ is said to be \emph{locally small} if the morphism class $\mathcal{C}(a,b)$ for each pair $(a,b)$ of objects in $v\mathcal{C}$ is a set instead of a proper class. Observe that natural categories like the category $\mathbf{Set}$ of all sets with functions as morphisms, the category $\mathbf{Grp}$ of all groups with group homomorphisms as morphisms, the category $\mathbf{RS}$ of all regular semigroups with semigroup homomorphisms as morphisms etc are locally small categories. 
	
	A morphism $f\colon c\to d$ is called an \emph{epimorphism} if it is left-cancellative, i.e., for all morphisms $h,k \colon d \to e$, the equality $f h = f k$ implies $h=k$. Similarly, a morphism $f\colon c\to d$ is called a \emph{monomorphism} if it is right-cancellative, i.e., for all morphisms $h,k \colon b \to c$, the equality $hf = kf$ implies $h=k$. A morphism $f\colon c\to d$ is called an \emph{isomorphism} if there exists a morphism $g\colon d \to c$ such that $fg=1_c$ and $gf=1_d$. Observe that an isomorphism is always an epimorphism and a monomorphism, but converse need not be true. There can be a morphism which is both an epimorphism and a monomorphism, but fails to be an isomorphism.
	
	A \emph{preorder} is a category with at most one morphism from an object to another. A \emph{strict preorder} is a preorder where the only isomorphisms are the identity morphisms. A small category in which every morphism is an isomorphism is called a \emph{groupoid}.
	
	\subsection{Functors}
	Given two categories $\mathcal{C}$ and $\mathcal{D}$, a \emph{functor} $F\colon\mathcal{C}\to\mathcal{D}$ consists of two functions: the object function (denoted by $vF$) which assigns to each object $a$ of $\mathcal{C}$, an object $vF(a)$ (often denoted by just $F(a)$) of the category $\mathcal{D}$ and the morphism function (denoted by $F$ itself) which assigns to each morphism $f\colon a\to b$ of $\mathcal{C}$, a morphism $F(f)\colon F(a) \to F(b)$ in $\mathcal{D}$. These two functions should respect identities and composition, that is, they should satisfy the following properties: $F(1_c)= 1_{F(c)}$ for every $c\in v\mathcal{C}$ and $F(f g) = F(f)  F(g)$, whenever $f g$ is defined in $\mathcal{C}$. In the sequel, although functors are written as left operators, the composition will be from left to right, i.e., if $F\colon \mathcal{C}\to \mathcal{D}$ and $G\colon\mathcal{D}\to \mathcal{E}$ are two functors then their composition functor $FG\colon\mathcal{C}\to \mathcal{E}$ shall be defined as $FG(c)= G(F(c))$ and $FG(f)= G(F(f))$ for each object $c\in v\mathcal{C}$ and for each morphism $f\in\mathcal{C}$.
	
	A functor $F\colon\mathcal{C}\to \mathcal{D}$ is said to be $v$-\emph{surjective}, $v$-\emph{injective} or $v$-\emph{bijective} if the object map $vF$ has the corresponding property. A functor $F\colon\mathcal{C}\to\mathcal{D}$ is said to be \emph{full}, \emph{faithful} or \emph{fully-faithful} if the morphism map $F$ is surjective, injective or bijective respectively. A functor $F\colon\mathcal{C}\to \mathcal{D}$ is said to be an \emph{isomorphism} if it is $v$-bijective and fully-faithful. 
	
	Given a category $\mathcal{C}$, a \emph{subcategory} of $\mathcal{C}$ is a category $\mathcal{D}$ whose objects are objects in $\mathcal{C}$ and whose morphisms are morphisms in $\mathcal{C}$, with the same identities and composition of morphisms. Given a category $\mathcal{C}$ with a subcategory $\mathcal{D}$, we define the inclusion functor $J\colon\mathcal{D}\to\mathcal{C}$ as follows:
	$$vJ(d)=d\quad\text{and}\quad J(g)=g$$
	for each object $d\in v\mathcal{D}$ and each morphism $g\in \mathcal{D}$. A subcategory $\mathcal{D}$ of a category $\mathcal{C}$ is said to be \emph{full} if the inclusion functor $J$ is full.
	
	If $F$ and $G$ are functors between the categories $\mathcal{C}$ and $\mathcal{D}$, then a \emph{natural transformation} $\sigma$ from $F$ to $G$ is a family of morphisms in $\mathcal{D}$ such that to every object $c$ in $\mathcal{C}$, we associate a morphism $\sigma(c) \colon F(c) \to G(c)$ called the \emph{component of $\sigma$ at $c$} so that the following diagram commutes for every $g\colon c\to d$ in $\mathcal{C}$:
	\begin{equation*}\label{natf}
	\xymatrixcolsep{4pc}\xymatrixrowsep{3pc}\xymatrix
	{
		F(c) \ar[r]^{\sigma(c)} \ar[d]_{F(g)}  
		& G(c) \ar[d]^{G(g)} \\       
		F(d) \ar[r]^{\sigma(d)} & G(d) 
	}
	\end{equation*}
	If every component $\sigma(c)$ is an isomorphism in $\mathcal{D}$, we say $\sigma$ is a \emph{natural isomorphism} and then the functors $F$ and $G$ are said to be \emph{naturally isomorphic}. 
	
	Natural transformations may be seen as morphisms between functors, and so the class of all functors between two categories, say $\mathcal{C}$ and $\mathcal{D}$, forms a category called the \emph{functor category} $[\mathcal{C},\mathcal{D}]$. Recall that $\mathbf{Set}$ stands for the category whose object class is the class of all sets and whose morphisms are all functions between sets. Given a category $\mathcal{C}$, we define a category $\mathcal{C}^*=[\mathcal{C},\mathbf{Set}]$ with the object class as the class of all functors from $\mathcal{C}$ to the category $\mathbf{Set}$ and natural transformations as morphisms.
	
	Given a locally small category $\mathcal{C}$, the set $\mathcal{C}(c,d)$ of morphisms between any two objects $c$ and $d$ in $\mathcal{C}$, gives rise to certain important functors in $\mathcal{C}^*$ called \emph{hom-functors}. For a fixed object $a\in v\mathcal{C}$, the hom-functor $\mathcal{C}(a,-)$ is defined as:
	\begin{itemize}
		\item the object function $v\mathcal{C}(a,-)$ maps each object $c\in v\mathcal{C}$ to the set of morphisms $\mathcal{C}(a,c)$;
		\item  the morphism function $\mathcal{C}(a,-)$ maps each morphism $f\in \mathcal{C}(c,d)$ to the function $\mathcal{C}(a,f)\colon\mathcal{C}(a,c)\to \mathcal{C}(a,d)$ given by $g\mapsto gf$ for each $g\in \mathcal{C}(a,c)$.
	\end{itemize}
	
	An arbitrary functor $F\colon\mathcal{C}\to \mathbf{Set}$ is said to be \emph{representable} if it is naturally isomorphic to the hom-functor $\mathcal{C}(a,-)$ for some object $a\in v\mathcal{C}$. Then we say that the object $a$ is the \emph{representing object of $F$}. The \emph{Yoneda lemma} states that the natural transformations $\mathcal{C}^*(\mathcal{C}(a,-),\mathcal{C}(b,-))$ between two hom-functors $\mathcal{C}(a,-)$ and $\mathcal{C}(b,-)$ are in one-to-one correspondence with the morphisms (in the reverse direction) between the associated objects, i.e., with the set $\mathcal{C}(b,a)$. Hence, the above discussion may be used to characterise morphisms between representable functors in terms of the morphisms between the representing objects.
	
	\subsection{Biordered sets}
	
	As mentioned in the Introduction, Nambooripad's initial approach \cite{kssthesis,ksssf,ksssf1,mem} to study a regular semigroup $S$ relied heavily on the information about the semigroup captured by its set of idempotents $E(S)$. Observe that for $e,f \in E(S)$, we can define a quasi-orders $\omega^r$ and $\omega^l$ on $E(S)$ as follows:
	$$e \omega^l f\iff ef=e \iff Se \subseteq Sf; \text{  and  } e \omega^r f \iff fe=e \iff eS \subseteq fS.$$
	Then clearly the restrictions of the Green relations on the idempotents of the semigroup are given by $\mathscr{L}=\omega^l\cap(\omega^l)^{-1}$ and $\mathscr{R}=\omega^r\cap(\omega^r)^{-1}$. Also the natural partial order $\omega$ on $E(S)$ is given by $\omega=\omega^l\cap\omega^r$.
	
	Nambooripad showed that the set of idempotents of a (regular) semigroup has an inherent structure of a \emph{(regular) biordered set}. A biordered set $E$ was abstractly axiomatised as a partial algebra\footnote{The original definition of a biordered set by Nambooripad \cite{kssthesis} was not as a partial algebra, instead it relied on a family of partial translations. Later Clifford \cite{clifbi} showed that biordered sets can be replaced by partial bands and subsequently in \cite{mem}, the partial algebra approach was adopted.} (i.e., a set with a partial binary composition defined on it) whose partial binary composition was determined by the two quasi-orders $\omega^l$ and $\omega^r$. The domain $D_E$ of the partial composition of a biordered set was given by $$D_E= (\omega^l\cup\omega^r)\cup(\omega^l\cup\omega^r)^{-1}$$ 
	and it satisfied certain axioms \cite{mem}.

	Given two biordered sets $E$ and $E'$ with the domain of partial compositions $D_E$ and $D_{E'}$ respectively, we can define a \emph{bimorphism} as a mapping $\theta\colon E\to E'$ satisfying:
	\begin{enumerate}
		\item [(BM1)] $(e,f)\in D_E \implies (e\theta,f\theta)\in D_{E'}.$
		\item [(BM2)] $(ef)\theta = (e\theta)(f\theta).$ 
	\end{enumerate}
	
	Further, for any pair of elements $e,f$ in a biordered set $E$, there is a quasiordered set $M(e,f) =\omega^l(e)\cap\omega^r(f)$ with the relation $\preceq$ as follows:
	$$g\preceq h \iff eg\omega^r eh\text{ and }gf\omega^l hf.$$
	Then the \emph{sandwich set} $\mathcal{S}(e,f)$ of $e$ and $f$ is defined as the set of maximal elements of $M(e,f)$ with respect to the quasiorder $\preceq$, i.e.,
	$$\mathcal{S}(e,f)=\{h\in M(e,f) : g\preceq h \text{ for all } g\in M(e,f) \}.$$
	
	The biordered set $E$ is called regular if $\mathcal{S}(e,f)$ is non empty. The sandwich set $\mathcal{S}(e,f)$ in a biordered set may be seen as the regular semigroup analog of the meet $e\wedge f$ of any two idempotents in the semilattice of an inverse semigroup. Thus Nambooripad was able to generalise Munn's construction of fundamental inverse semigroups to fundamental regular semigroups by replacing semilattices with regular biordered sets. Later in 1986, Easdown \cite{eas} showed that arbitrary (not necessarily regular) biordered sets come from semigroups, in the sense that given a biordered set $E$, there exists a semigroup $S$ such that the biordered set of $S$ is biorder isomorphic to $E$. This has unequivocally established the importance of biordered sets in the structure of semigroups.

	\subsection{Inductive groupoids and structure mappings}
	As discussed earlier, Nambooripad \cite{kssthesis,ksssf1} had also extended Schein's approach of inverse semigroups based on ordered groupoids, to construct arbitrary regular semigroups from regular biordered sets and groupoids.
	\begin{dfn}\label{og}
		Let $\mathcal{G}$ be a groupoid and $\leq$ a partial order on $\mathcal{G}$.  Let $e,f \in v\mathcal{G}$ and $x,y$ etc denote arbitrary morphisms of $\mathcal{G}$ such that $\mathbf{d}(x)$ and $\mathbf{r}(x)$ is the domain and codomain respectively of an arbitrary morphism $x$. Then $(\mathcal{G},\leq)$ is called an \emph{ordered groupoid} if the following hold.
		\begin{enumerate}
			\item [(OG1)] If $u\leq x$, $v\leq y$ and $\mathbf{r}(u)=\mathbf{d}(v)$, $\mathbf{r}(x)=\mathbf{d}(y)$, then $uv \leq xy$.
			\item [(OG2)] If $x\leq y$, then $x^{-1}\leq y^{-1}$.
			\item [(OG3)] If $1_e\leq 1_{\mathbf{d}(x)}$, then there exists a unique element $e{\downharpoonleft} x$ (called the \emph{restriction} of $x$ to $e$) in $\mathcal{G}$ such that $e{\downharpoonleft} x\leq x$ and $\mathbf{d}(e{\downharpoonleft} x) = e$.
			\item [(OG3$^*$)] If $1_f\leq 1_{\mathbf{r}(x)}$, then there exists a unique element $x{\downharpoonright} f$ (called the \emph{corestriction} of $x$ to $f$) in $\mathcal{G}$ such that $x{\downharpoonright} f \leq x$ and $\mathbf{r}(x{\downharpoonright} f) = f$.
		\end{enumerate}
	\end{dfn}
	
	Observe that in the above definition, the axiom (OG3$^*$) is the dual of the axiom (OG3). In fact it can be shown that if $(\mathcal{G},\leq)$ is a groupoid satisying axioms (OG1) and (OG2), then  $(\mathcal{G},\leq)$ satisfies (OG3) if and only if $(\mathcal{G},\leq)$ satisfies (OG3$^*$).	
	
	A functor $F$ between two ordered groupoids is said to be a $v$-isomorphism if the object map $vF$ is an order isomorphism. 
	
	The ordered groupoids of inverse semigroups are completely characterised by the property that the partially ordered set of identities (equivalently the set of objects) forms a semilattice \cite{lawson}. But even if we replace the set of objects of an ordered groupoid with a regular biordered set, it is not yet sufficient to construct arbitrary regular semigroups. As remarked in \cite{knr}, this is because the global structure of the semigroup is still not sufficiently reflected on the groupoid. So, Nambooripad added an additional layer of (biorder) structure to this groupoid as follows.
	
	Given a regular biordered set $E$, Nambooripad defined an \emph{E-path} as a sequence of elements $( e_1, e_2 , \dotsc , e_n )$ of $E$ such that $e_i ( \mathscr{R}\cup\mathscr{L} ) e_{i + 1}$ for all $i = 1 , \dots , n-1$. An idempotent $e_i$ in an E-path $( e_1 , e_2 ,\dots , e_n )$ is \emph{inessential} if $e_{i-1}\mathscr{R} e_i \mathscr{R} e_{i + 1}$ or $e_{i-1}\mathscr{L} e_i \mathscr{L} e_{i + 1}$. We can define an equivalence relation on the set of E-paths by adding or removing inessential vertices. The equivalence class of an E-path relative to this equivalence relation is defined to be an \emph{E-chain} and it can be seen that every E-chain has a unique canonical representative of the form $\mathbf{c}( e_1 , e_2 , \dots , e_n )$ where each vertex is essential. 
	
	For instance, consider the following example. In the sequel, we shall often use diagrams (as below) to represent the structural scenario in semigroups or biordered sets, wherein a horizontal line represents the Green $\mathscr{R}$ relation and a vertical line represents the Green $\mathscr{L}$ relation. In Figure \ref{figechain}, the elements of a biordered set $E$ are represented using dots. Here the elements $e_2$ and $e_4$ are inessential in the E-path $(e_1,e_2,e_3,e_4,e_5)$ of $E$. So the canonical representation of the corresponding E-chain is $\mathbf{c}(e_1,e_3,e_5)$. In the sequel whenever there is no risk of confusion, by abuse of notation we shall denote an $E$-chain $\mathbf{c}( e_1 , e_2 , \dots , e_n )$ by just $( e_1 , e_2 , \dots , e_n )$.
	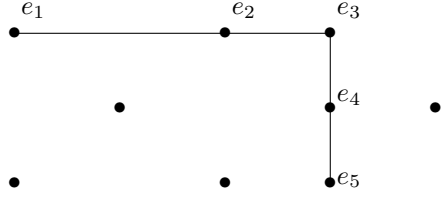
\begin{figure}
		\begin{center}
			\begin{tikzpicture}[node distance=0 cm,outer sep = 0pt]
			\tikzstyle{grp}=[draw, rectangle, minimum height=1cm, minimum width=1.4cm, fill=violet!0,draw=none,anchor=south west]
			\tikzstyle{dc}=[draw, rectangle, font={\tiny}, minimum height=1cm, minimum width=1.4cm, fill=violet!0,draw=none,anchor=south west]
			
			\node[grp] (s11) at (5,3) {$\bullet$};
			\node[dc] (s12) [right = of s11] {};
			\node[grp] (s13) [right = of s12] {$\bullet$};
			\node[grp] (s14) [right = of s13] {$\bullet$};
			\node[dc] (s15) [right = of s14] {};
			\node[dc] (s21) [below = of s11] {};
			\node[grp] (s22) [right = of s21] {$\bullet$};
			\node[dc] (s23) [below = of s13] {};
			\node[grp] (s24) [right = of s23] {$\bullet$};
			\node[grp] (s25) [right = of s24] {$\bullet$};
			\node[grp] (s31) [below = of s21] {$\bullet$};
			\node[dc] (s32) [right = of s31] {};
			\node[grp] (s33) [right = of s32] {$\bullet$};
			\node[grp] (s34) [right = of s33] {$\bullet$};
			\node[dc] (s35) [right = of s34] {};
			
			\node (e1) at (5.95,3.8) {$e_1$};
			\node (e2) at (8.75,3.8) {$e_2$};
			\node (e3) at (10.15,3.8) {$e_3$};
			\node (e4) at (10.15,2.65) {$e_4$};
			\node (e5) at (10.15,1.55) {$e_5$};
			
			\draw[-] (s11.center) to (s14.center); 
			\draw[-] (s14.center) to (s34.center);
			\end{tikzpicture}
			\caption{E-chain in a biordered set}\label{figechain}
		\end{center}
		
	\end{figure}
	
	The set $\mathcal{G} ( E )$ of E-chains forms a groupoid with set $E$ of objects and with an E-chain $ ( e_1 , e_2 , \dots , e_n )$ as a morphism from $e_1$ to $e_n$. The product of two E-chains $ ( e_1 , e_2 , \dots , e_n )$ and $ ( f_1 , f_2 , \dots , f_m )$ is defined if and only if $e_n = f_1$ and is equal to the canonical representative of $ ( e_1 , \dots , e_n = f_1 , \dots , f_m )$ . The inverse of $ ( e_1 , e_2 , \dots , e_n )$ is the $E$-chain  $ ( e_n , \dots , e_2 , e_1 )$. 
	
	Further for an $E$-chain $\mathfrak{c}=  (e_0,e_1,\dots,e_n)\in\mathcal{G} ( E )$ and $h\in \omega(e_0)$, let
	$$h\cdot \mathfrak{c}=  (h,h_0,h_1,\dots,h_n) \text{ where } h_0=he_0 \text{ and }h_i=e_ih_{i-1}e_i \text{ for } i=1,\dots,n.$$
	
	Then for $\mathfrak{c},\mathfrak{c}'\in \mathcal{G}(E)$, if we define a partial order on $\mathcal{G} ( E )$ as follows: 
	$$\mathfrak{c}\leq_E \mathfrak{c}' \iff \mathbf{d}(\mathfrak{c})\omega \mathbf{d}(\mathfrak{c}') \text{ and }\mathfrak{c}= \mathbf{d}(\mathfrak{c})\cdot \mathfrak{c}'$$ 
	and define restriction $h{\downharpoonleft}\mathfrak{c}=h\cdot \mathfrak{c}$ for all $h\in \omega(\mathbf{d}(\mathfrak{c}))$, then $(\mathcal{G} ( E ),\leq_E)$ forms an ordered groupoid called the \emph{groupoid of E-chains} of the biordered set $E$.
	
	Given a biordered set $E$, a $2 \times 2$ matrix $\bigl[ \begin{smallmatrix} e&f\\ g&h \end{smallmatrix} \bigr]$
	of elements of $E$ such that $e\mathscr{R} f\mathscr{L}h\mathscr{R} g\mathscr{L}e$ is known as an \emph{E-square}. Observe that using the convention introduced in Figure \ref{figechain}, an E-square will in fact be represented by a square. 
	
	An E-square of the form  $\bigl[ \begin{smallmatrix} g&h\\ eg&eh \end{smallmatrix} \bigr]$  
	where $g,h \in \omega^l(e)$ and $g\mathscr{R} h$ is said to be row-singular. Dually, an E-square $\bigl[ \begin{smallmatrix} g&ge\\ h&he \end{smallmatrix} \bigr]$
	is said to be column-singular if $g,h \in \omega^r(e)$ and $g\mathscr{L} h$. An E-square is said to be \emph{singular} if it is either row-singular or column-singular.

	Let $E$ be a regular biordered set and $\epsilon\colon \mathcal{G}(E)\to \mathcal{G}$ be a $v$-isomorphism of $\mathcal{G}(E)$ into an ordered groupoid $\mathcal{G}$. Then an E-square $\bigl[ \begin{smallmatrix} e&f\\ g&h \end{smallmatrix} \bigr]$
	is said to be $\epsilon$-commutative if 
	$$\epsilon(e,f)\epsilon(f,h) =\epsilon(e,g)\epsilon(g,h).$$
	Observe that in the above statement to simplify notation, given an $E$-chain $(e,f)$ we have used $\epsilon(e,f)$ instead of $\epsilon((e,f))$. This convention shall be followed in the sequel.
	
	Using this additional layer of structure provided by the ordered groupoid $\mathcal{G}(E)$ of E-chains and a $v$-isomorphism  $\epsilon\colon \mathcal{G}(E)\to \mathcal{G}$ called an \emph{evaluation functor} from $\mathcal{G}(E)$ into an ordered groupoid $\mathcal{G}$, Nambooripad defined an inductive groupoid as follows.
	\begin{dfn}\label{dfnig}
		Let $E$ be a regular biordered set and $\epsilon $ be an evaluation functor from $\mathcal{G}(E)$ into an ordered groupoid $\mathcal{G}$. We say that $(\mathcal{G},\epsilon )$ forms an \emph{inductive groupoid} if the following axioms and their duals hold.
		\begin{enumerate}
			\item[(IG1)] Let $x\in \mathcal{G}$ and for $i=1,2$, let $e_i$, $f_i \in E$ such that $\epsilon (e_i) \leq \mathbf{d}(x)$ and $\epsilon (f_i) = \mathbf{r}(\epsilon (e_i){\downharpoonleft} x)$. If $e_1\omega^r e_2$, then $f_1\omega^r f_2$, and
			$$\epsilon (e_1,e_1e_2)(\epsilon (e_1e_2){\downharpoonleft} x) = (\epsilon (e_1){\downharpoonleft} x)\epsilon (f_1,f_1f_2).$$
			\item[(IG2)] All singular E-squares are $\epsilon $-commutative.
		\end{enumerate}
	\end{dfn}
	
	Let $(\mathcal{G},\epsilon )$ and $(\mathcal{G}',\epsilon ')$ be two inductive groupoids with biordered sets $E$ and $E'$ respectively. An order preserving functor $F\colon \mathcal{G} \to\mathcal{G}'$ is said to be \emph{inductive} if $vF\colon E \to E'$ is a bimorphism of biorder sets such that the following diagram commutes.
	\begin{equation}\label{indf}
	\xymatrixcolsep{4pc}\xymatrixrowsep{3pc}\xymatrix
	{
		\mathcal{G}(E) \ar[r]^{\mathcal{G}(vF)} \ar[d]_{\epsilon }  
		& \mathcal{G}(E') \ar[d]^{\epsilon '} \\       
		\mathcal{G} \ar[r]^{F} & \mathcal{G}' 
	}
	\end{equation}
	
	\subsubsection{Inductive groupoid of a regular semigroup} \label{sssecindrs}
	Nambooripad showed that given a regular semigroup $S$ with biordered set $E$, we can associate an inductive groupoid $(\mathcal{G}(S),\epsilon_S)$ as follows. The set of objects $v\mathcal{G}(S)=E$ and the set of morphisms 
	$$\mathcal{G}(S)=\{(x,x') : x\in S \text{ and } x'\in V(x)\} $$ 
	where recall $V(x)$ denotes the set of inverses of $x$. For $(x, x'), (y,y') \in \mathcal{G}(S)$, the composition is defined
	$$(x, x')(y, y') = (x y, y'x')\text{ if }x'x = yy'. $$
	For $(x, x') \in \mathcal{G}(S)$, $\mathbf{d} (x,x') = (xx', xx ')$ is the left identity, $\mathbf{r}(x,x') = (x'x, x'x)$ is the right identity and
	$(x', x)$ is the inverse. The restriction of $(x,x')$ to $e\in \omega(xx')$ is defined as $e{\downharpoonleft}(x,x')=(ex,x'e)$ and the corestriction of $(x,x')$ to $f\in \omega(x'x)$ is defined as $(x,x'){\downharpoonright} f=(xf,fx')$. The partial order $\leq_\mathcal{G}$ in $\mathcal{G}(S)$ is given by:
	$$(x,x')\leq_\mathcal{G} (y,y') \iff x=(xx')y, x'=y'(xx') \text{ and } xx'\omega yy'.$$
	This makes $\mathcal{G}(S)$ an ordered groupoid. Since $E$ is a regular biordered set, the groupoid $\mathcal{G}(E)$ of $E$-chains is an ordered groupoid. So, the evaluation functor $\epsilon_S\colon \mathcal{G}(E)\to \mathcal{G}(S)$ is defined as follows. The object map is $v\epsilon_S = 1_E$ and for each E-chain $\mathfrak{c} =  (e_0 , e_1 ,\dotsc e_n )\in \mathcal{G}(E)$ from $e_0$ to $e_n$,
	$$\epsilon_S (\mathfrak{c}) = (e_0 e_1 \dots e_{n-1} e_n ,\: e_n e_{n-1} \dotsc e_{1} e_0).$$
	
	Then it may be verified that $(\mathcal{G}(S),\epsilon_S)$ is an inductive groupoid.
	
	\subsubsection{Regular semigroup of an inductive groupoid} Conversely given the inductive groupoid $(\mathcal{G},\epsilon_\mathcal{G})$, we can define an equivalence relation $p$ on the set of morphisms $\mathcal{G}$  as follows. Given $x ,y \in \mathcal{G}$, 
	$$x \:p\:y  \iff \mathbf{d}(x )\mathscr{R}\mathbf{d}(y ),\:\mathbf{r}(x )\mathscr{L}\mathbf{r}(y )\:\text{ and } x \:\epsilon_\mathcal{G}(\mathbf{r}(x ),\mathbf{r}(y )) = \epsilon_\mathcal{G}(\mathbf{d}(x ),\mathbf{d}(y ))\:y .$$
	
	Then the $p$-classes of morphisms in $\mathcal{G}$ form a regular semigroup under the binary composition as described below. If $\bar{x}$ and $\bar{y}$ denote the $p$-classes of $\mathcal{G}$ containing the morphisms $x$ and $y$ respectively, then for a sandwich element $h\in \mathcal{S}(\mathbf{r}(x),\mathbf{d}(y))$, the binary composition is defined as:
	\begin{equation}\label{eqnbincomp}
	\bar{x}. \bar{y} = \overline{(x {\downharpoonright} \mathbf{r}(x)h)\:\epsilon(\mathbf{r}(x)h,h)\:\epsilon(h,h\mathbf{d}(y))\:(h\mathbf{d}(y) {\downharpoonleft} y)}.
	\end{equation}
	
	The binary composition in (\ref{eqnbincomp}) can be illustrated using Figure \ref{figind} wherein arcs denote the morphisms in the inductive groupoid $\mathcal{G}$. 
	
	This gives the inductive groupoid representation of regular semigroups as described in \cite{mem}. The original description in \cite{kssthesis,ksssf1} of the same idea was using the notion of a 
	\emph{regular groupoid} of a regular semigroup and certain mappings between $\mathscr{L}$-classes and $\mathscr{R}$-classes. 
	It was exactly these mappings which Meakin had called \emph{structure mappings} 
	and used extensively in \cite{stmp0,stmp,stmp1}. One can see that the restrictions and corestrictions  are nothing but reincarnations of the structure mappings.
	
	Thus from the family of structure mappings and the {regular groupoid} associated with a regular semigroup, which Nambooripad collectively called (and axiomatised) as the \emph{regular system}, one could equivalently retrieve the regular semigroup. 		
	
	\begin{figure}
		\centering
		\xymatrixcolsep{1.5pc}\xymatrixrowsep{1.5pc}
		\xymatrix{ 
			\ar@{-}@/^1pc/[rrr]^{x}&&&\mathbf{r}(x) \ar@{.}[ddddr]^{\omega}&&&\mathbf{d}(y) \ar@{.}[dddl]_{\omega}\ar@{-}@/^1pc/[rrr]^{y}&&& \\\\\\
			&&&&h\ar@{-}[r]&h\mathbf{d}(y)\ar@{-}@/^1pc/[rrr]^{h\mathbf{d}(y) {\downharpoonleft} y}&&& \\
			&\ar@{-}@/^1pc/[rrr]^{x {\downharpoonright} \mathbf{r}(x)h}&&&\mathbf{r}(x)h\ar@{-}[u]_{}
		}
		\caption{Binary composition in an inductive groupoid}\label{figind}
	\end{figure}
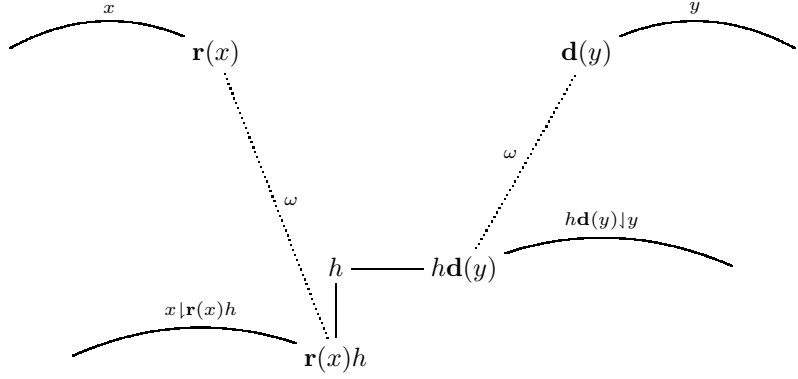
	
	Using this correspondence between groupoids and semigroups, Nambooripad \cite{mem} explicitly showed that the category $\bf{IG}$ of inductive groupoids is equivalent to the category $\bf{RS}$ of regular semigroups. The ESN Theorem may be seen as the special case of this when specialised to inverse semigroups.
	
	\subsection{Regular partially ordered sets and Cross-connections}
	
	In 1973, Grillet \cite{gril,gril1,gril2} constructed fundamental regular semigroups from a cross-connected pair of partially ordered sets. In the process, he characterised the partially ordered sets of regular semigroups as regular partially ordered sets using the idea of a normal mapping. We briefly recall some crucial definitions of Grillet.
	\begin{dfn}\label{dfnnm}
		Let $X=(X,\leq)$ be a partially ordered set  and for $x\in X$, let $X(x)=\{y\in X: y\leq x\}$ be the principal order ideal of $X$ generated by $x$. A mapping $\alpha\colon X\to X$ is said to be a \emph{normal mapping} if:
		\begin{enumerate}
			\item [(Nmap1)] The range Im $\alpha$ of the mapping $\alpha$ is a principal ideal of $X$.
			\item [(Nmap2)] The mapping $\alpha$ is order preserving.
			\item [(Nmap3)] For each $x\in X$, there exists $y\leq x$ such that $\alpha$ maps $X(y)$ isomorphically upon $X(\alpha x)$.
		\end{enumerate}
	\end{dfn}
	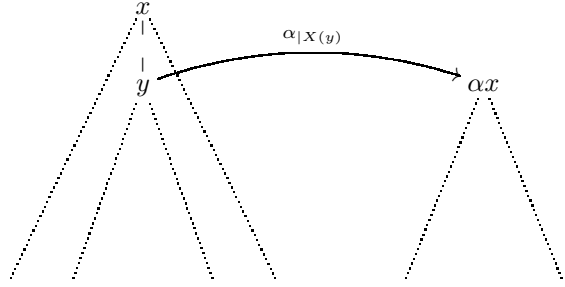
\begin{figure}
		\begin{center}
			\xymatrixcolsep{1.5pc}\xymatrixrowsep{1.5pc}
			\xymatrix{ 
				&&&&&x \ar@{.}[ddddll]^{}\ar@{.}[ddddrr]^{}\ar@{--}[d]^{}\\
				&&&&&y \ar@{.}[dddl]^{}\ar@{.}[dddr]^{} \ar@{->}@/^1pc/[rrrrr]^{\alpha_{|X(y)}}&&&&& \alpha x \ar@{.}[dddl]^{}\ar@{.}[dddr]^{}\\
				\\
				&&&& && &&&  & &\\
				&&& & & & & & & & &\\}
			\caption{Isomorphism in a normal mapping}\label{fignm}
		\end{center}
	\end{figure}
	
	Figure \ref{fignm} represents the condition (Nmap3) of a normal mapping; this property will be crucial in the sequel.	The \emph{apex} of the normal mapping $\alpha$ is the element $a$ of $X$ such that Im $\alpha=X(a)$. An idempotent normal mapping is called a \emph{normal retraction}. A partially ordered set $X$ is said to be \emph{regular} if every element of $X$ is the apex of a normal retraction of $X$. The set $N(X)$ of normal mappings on a regular partially ordered set $X$ forms a regular semigroup.
	
	Further, an equivalence relation on a partially ordered set $X$ is \emph{normal} if it is induced by a normal mapping. The set of all normal equivalence relations on $X$ ordered by reverse inclusion forms the regular partially ordered set $X^*$, known as the \emph{dual} of $X$.
	
	A \emph{cross-connection} between two regular partially ordered sets $I$ and $\Lambda$ is a pair $(\Gamma,\Delta)$ of order preserving mappings $\Gamma\colon \Lambda\to I^*$ and $\Delta\colon I\to \Lambda^*$ satisfying certain axioms \cite{gril1}. 
	
	Grillet showed that given a fundamental regular semigroup, it induces a cross-connection between its partially ordered sets of principal left and right ideals, and conversely every cross-connection gives rise to a fundamental regular semigroup.
	
	Nambooripad's article \cite{bicxn} implied that isomorphic biordered sets determine isomorphic cross-connections and hence Grillet's construction is insufficient to characterize arbitrary regular semigroups. Later, Nambooripad observed that a partially ordered set can be seen as a strict preorder category. Elaborating this idea, he \cite{cross0} replaced Grillet's regular partially ordered sets with \emph{normal categories}, and constructed arbitrary regular semigroups as cross-connection semigroups. 
	
	\subsection{Nambooripad's Cross-connections} \label{cxns} Now we briefly recall Nambooripad's construction \cite{cross}. In the sequel, unless otherwise stated the term ``cross-connection'' shall refer to Nambooripad's generalised version with categories.  
	
	Let $\mathcal{C}$ be a small category and $\mathcal{P}$ be a subcategory of $\mathcal{C}$ such that $\mathcal{P}$ is a strict preorder with $v\mathcal{P} = v\mathcal{C}$. Then $(\mathcal{C} ,\mathcal{P})$ is called a \emph{category with subobjects}\index{category!with subobjects} if, first, every $f\in \mathcal{P}$ is a monomorphism in $\mathcal{C}$, second, and if $f=hg$ for $f,g\in \mathcal{P}$ and $h \in \mathcal{C}$, then $h\in \mathcal{P}$.
	
	In a category $(\mathcal{C} ,\mathcal{P})$ with subobjects, the morphisms in $ \mathcal{P}$ are called inclusions. If $c' \to c$ is an inclusion, we write $c'\subseteq c$ and we denote this inclusion by $j(c',c)$. An inclusion $j(c',c)$ \emph{splits} if there exists $q\colon c\to c' \in \mathcal{C}$ such that $j(c',c)q =1_{c'}$. Then the morphism $q$ is called a \emph{retraction}. 
	
	A \emph{normal factorization} of a morphism $f \in \mathcal{C}(c,d)$ is a
	factorization of the form $f=quj$ where $q\colon c\to c'$ is a retraction,
	$u\colon c'\to d'$ is an isomorphism and $j=j(d',d)$ an inclusion where $c',d' \in v\mathcal{C}$ with $ c' \subseteq c$, $ d' \subseteq d$. The morphism $qu$ is known as the \emph{epimorphic component} of the morphism $f$ and is denoted by $f^\circ$. Figure \ref{fignf} represents the normal factorisation property. Compare it with Figure \ref{fignm} to observe that normal factorisation is a generalisation of the property (Nmap3) of a normal mapping.
	
	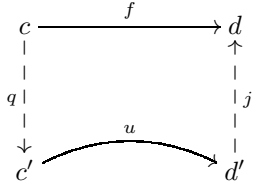
\begin{figure}
		\centering
		\xymatrixcolsep{1.5pc}\xymatrixrowsep{1.5pc}
		\xymatrix{ 
			&&&&&c \ar@{->}[rrr]^{f} \ar@{-->}[dd]_{q}&&&d\\
			\\
			&&&&&c' \ar@{->}@/^1pc/[rrr]^{u} &&&d'\ar@{-->}[uu]_{j}
		}
		\caption{Normal factorisation of a morphism $f$}\label{fignf}
	\end{figure}
	
	\begin{dfn}\label{dfnnc}
		Let $\mathcal{C}$ be a category with subobjects and $d\in v\mathcal{C}$. A map $\gamma\colon v\mathcal{C}\to\mathcal{C}$ is called a \emph{normal cone} from the base $v\mathcal{C}$ to the apex\footnote{In \cite{cross}, the terminogy used for apex is ``vertex'' but we avoid it as it may lead to some confusion with the vertices of a category.} $d$ if 
		\begin{enumerate}
			{
				\item[(Ncone1)] $\gamma(c)\in \mathcal{C}(c,d)$ for all $c\in v\mathcal{C}$. 
				\item[(Ncone2)] If $c\subseteq c'$ then $j(c,c')\gamma(c') = \gamma(c)$. 
				\item[(Ncone3)] There exists $c\in v\mathcal{C}$ such that $\gamma(c)\colon c\to d$ is an isomorphism.
			}
		\end{enumerate}
	\end{dfn}
	
	Given a normal cone $\gamma$, we denote by $c_{\gamma}$ the {apex} of $\gamma$ and for each $c\in v\mathcal{C}$, the morphism $\gamma(c)\colon  c \to c_{\gamma}$ is called the {component} of $\gamma$ at $c$.  
	\begin{dfn}
		A category $\mathcal{C}$ with subobjects is called a \emph{normal category} if the following holds.
		\begin{enumerate}
			\item [(NC1)]Any morphism in $\mathcal{C}$ has a normal factorization. 
			\item [(NC2)]Every inclusion in $\mathcal{C}$ splits.
			\item [(NC3)]For each $c \in v\mathcal{C} $ there is a normal cone $\gamma$ with apex $c$ and $\gamma (c) = 1_c$.
		\end{enumerate}
	\end{dfn}
	Observe that given a normal cone $\gamma$ and an epimorphism $f\colon c_\gamma \to d$, the map $\gamma*f \colon  a \mapsto\gamma(a)f$ from $v\mathcal{C}$ to $\mathcal{C}$ is a normal cone with apex $d$. Hence, given two normal cones $\gamma$ and $\sigma $, we can compose them as follows.
	\begin{equation} \label{eqnsg1}
	\gamma \cdot \sigma = \gamma*(\sigma(c_\gamma))^\circ 
	\end{equation} 
	where $(\sigma(c_\gamma))^\circ$ is the epimorphic part of the morphism $\sigma(c_\gamma)$. All the {normal cones} in a normal category $\mathcal{C}$ with this special binary composition form a regular semigroup known as the \emph{semigroup of normal cones} in $\mathcal{C}$ and is denoted by $T\mathcal{C}$. It can be easily seen that a normal cone $\gamma\in T\mathcal{C}$ is an {idempotent} if and only if $\gamma(c_\gamma)= 1_{c_\gamma}$.
	
	To describe cross-connections, Grillet \cite{gril} used the set of all {normal} equivalence relations on a {regular} partially ordered set. To extend this idea to normal categories, Nambooripad proposed the notion of a {normal dual}. The \emph{normal dual} $N^\ast\mathcal C$ of a normal category $\mathcal{C}$ is a full subcategory of the category $\mathcal{C}^\ast$ of all functors from $\mathcal C$ to $\bf{Set}$. The objects of $N^\ast\mathcal C$ are functors called $H$-functors and the morphisms are natural transformations between them. 
	
	For each $\gamma \in T\mathcal{C}$, the \emph{H-functor} $H({\gamma};-)\colon  \mathcal{C}\to \mathbf{Set}$ is defined as follows. For each $c\in v\mathcal{C}$ and for each $g\in \mathcal{C}(c,d)$,
	\begin{subequations} \label{eqnH11}
		\begin{align}
		H({\gamma};{c})&= \{\gamma\ast f^\circ : f \in \mathcal{C}(c_{\gamma},c)\} \text{ and }\\
		H({\gamma};{g}) \colon H({\gamma};{c}) &\to H({\gamma};{d}) \text{ given by }\gamma\ast f^\circ \mapsto \gamma\ast (fg)^\circ.
		\end{align}
	\end{subequations}
	We define the \emph{M-set} of a normal cone $\gamma$ as 
	\begin{equation*}
	M\gamma = \{ c \in \mathcal{C}:\gamma(c)\text{ is an isomorphism} \}.
	\end{equation*} 
	It can be shown that if $H({\gamma};-) = H({\gamma'};-)$, then the $M$-sets of the normal cones $\gamma$ and $\gamma'$ coincide; and hence we define the $M$-set of an $H$-functor as $MH(\gamma;-) = M\gamma$.
	
	It can be seen that $H$-functors are {representable functors} such that for a normal cone $\gamma$ with apex $d$, there is a natural isomorphism $\eta_\gamma\colon  H(\gamma;-) \to \mathcal{C}(d,-)$. Here $\mathcal{C}(d,-)$ is the hom-functor determined by $d\in v\mathcal{C}$. 
	
	An \emph{ideal} $\langle c \rangle$ of a normal category $\mathcal{C}$ is the full subcategory of $\mathcal{C}$ whose objects are given by
	$$v\langle c \rangle =\{d\in v\mathcal{C}: d\subseteq c\}.$$
	
	A functor $F$ between two normal categories $\mathcal{C}$ and $\mathcal{D}$ is said to be a \emph{local isomorphism}\index{local isomorphism} if $F$ is inclusion preserving, fully faithful and for each $c \in v\mathcal{C}$, $F_{|\langle c \rangle}$ is an isomorphism of the ideal $\langle c \rangle$ onto $\langle F(c) \rangle$.
	\begin{dfn} \label{ccxn}
		Let $\mathcal{C}$ and $\mathcal{D}$ be normal categories. A \emph{cross-connection}\index{cross-connection} from $\mathcal{D}$ to $\mathcal{C}$ is a triplet $(\mathcal{D},\mathcal{C};{\Gamma})$ where $\Gamma\colon  \mathcal{D} \to N^\ast\mathcal{C}$ is a local isomorphism such that for every $c \in v\mathcal{C}$, there is some $d \in v\mathcal{D}$ such that $c \in M\Gamma(d)$.
	\end{dfn}
	In the sequel, when there is no ambiguity we shall often refer to a cross-connection $(\mathcal{D},\mathcal{C};{\Gamma})$ by just $\Gamma$. Then the set $E_\Gamma$ is defined by
	$$E_\Gamma = \{ (c,d) \in v\mathcal{C}\times v\mathcal{D} \text{ such that } c\in M\Gamma(d) \}.$$
	
	Given a cross-connection $(\mathcal{D},\mathcal{C};{\Gamma})$, there always exists a unique \emph{dual cross-connection} $\Delta$ from $\mathcal{C}$ to $N^\ast\mathcal{D}$. It may be seen that $(c,d) \in E_\Gamma$ if and only $(d,c) \in E_\Delta$. Also, if $(c,d)\in E_\Gamma$, then the unique idempotent cone $\epsilon \in T\mathcal{C}$ such that $c_\epsilon=c$ and $H(\epsilon;-)=\Gamma(d)$ shall be denoted by $\gamma(c,d)$. Similarly $\delta(c,d)$ shall denote the unique idempotent cone in $T\mathcal{D}$ when $(d,c)\in E_\Delta$. Further by suitably defining the basic products and quasi-orders, the set $E_\Gamma$ can be shown to be the regular biordered set associated with the cross-connection $\Gamma$.
	
	Given a cross-connection $(\mathcal{D},\mathcal{C};{\Gamma})$ with dual $\Delta$, let $(c,d),(c',d') \in E_\Gamma$, $f\in \mathcal{C}(c,c')$ and $g\in \mathcal{D}(d',d)$. Then $f$ is called the \emph{transpose} of $g$ from $c$ to $c'$, if $f$ and $g$ makes the following diagram commute:
	\begin{equation*}\label{trans}
	\xymatrixcolsep{2pc}\xymatrixrowsep{3pc}\xymatrix
	{
		\Delta(c) \ar[rr]^{\eta_{\delta(c,d)}} \ar[d]_{\Delta(f)}  
		& & \mathcal{D}(d,-) \ar[d]^{\mathcal{D}(g,-)}& d \\       
		\Delta(c') \ar[rr]^{\eta_{\delta(c',d')}} & & \mathcal{D}(d',-)&d'\ar[u]_{g} 
	}
	\end{equation*}
	
	\begin{dfn}\label{mor}
		Let $(\mathcal{D},\mathcal{C};{\Gamma})$ and $(\mathcal{D}',\mathcal{C}';{\Gamma}')$ be two cross-connections. A \emph{morphism of cross-connections} $m\colon \Gamma\to \Gamma'$ is a pair $m=(F_m,G_m)$ of inclusion preserving functors $F_m\colon \mathcal{C}\to \mathcal{C}'$ and $G_m\colon \mathcal{D}\to \mathcal{D}'$ satisfying the following axioms:
		\begin{enumerate}
			\item [(M1)] $(c,d)\in E_\Gamma \implies (F_m(c),G_m(d)) \in E_{\Gamma'}$ and for all $c'\in v\mathcal{C}$,
			$F_m(\gamma(c,d)(c'))=\gamma(F_m(c),G_m(d))(F_m(c'))$ .
			\item [(M2)] If $(c,d), (c',d') \in E_\Gamma$ and if $f^*\colon d'\to d$ is the transpose of $f\colon c\to c'$, then $G_m(f^*) = (F_m(f))^*$.
		\end{enumerate}  
	\end{dfn}
	\subsubsection{Cross-connections of a regular semigroup}\label{cxnrs} Given a regular semigroup $S$ with biordered set $E$, we can associate a normal category $\mathcal{L}_S$ of principal left ideals as follows. An object of the category $\mathcal{L}_S$ is a principal left ideal $Se$ for $e\in E$, and a morphism from $Se$ to $Sf$ is a partial right translation $\rho(e,u,f)\colon  u \in eSf$. That is, for $x\in Se$, the morphism $\rho(e,u,f)\colon x \mapsto xu \in Sf$. Two morphisms $\rho(e,u,f)$ and $\rho(g,v,h)$ are equal in $\mathcal{L}_S$ if and only if $e \mathscr{L} g$, $f\mathscr{L} h$, $u \in eSf$, $v\in gSh$ and $u=ev$.
	
	Given an arbitrary element $a \in S$, it induces certain \emph{principal cones} $\rho^a$ with apex $Sa$ whose component at any $Se \in v\mathcal{L}_S$ is given by $\rho^a(Se) = \rho(e,ea,f) $ where $f \in E(L_a)$. Then the $M$-set of $\rho^a$,
	$$M\rho^a=\{Se:e\in E(R_a)\}.$$
	
	Dually, we have a normal category $\mathcal{R}_S$ of principal right ideals of the semigroup $S$ such that an object of the category $\mathcal{R}_S$ is a principal right ideal $eS$, and the morphisms are partial left translations $\lambda(e,w,f)$ such that $w \in fSe$. The principal cones in $\mathcal{R}_S$ are given by $\lambda^a(eS) = \lambda(e,ae,f) $ where $f \in E(R_a)$.  Given two morphisms $\lambda(e,u,f)$ and $\lambda(g,v,h)$ in the category $\mathcal{R}_S$, they are equal if and only if $e \mathscr{R} g$, $f\mathscr{R} h$, $u \in fSe$, $v\in hSg$ and $u=ve$.
	
	Further, the following theorem describes the explicit relationship between the normal categories of a regular semigroup $S$: the categories $\mathcal{L}_S$ and $\mathcal{R}_S$ are cross-connected by a functor $\Gamma_S$.
	\begin{thm}\cite[Theorem IV.17]{cross}\label{thmcxn}
		The functor $\Gamma_S\colon  \mathcal{R}_S \to N^*\mathcal{L}_S$ defined by
		\begin{equation} 
		v\Gamma_S(eS) = H(\rho^e;-) \quad\text{ and }\quad \Gamma_S(\lambda(e,u,f)) = \eta_{\rho^e}\mathcal{L}_S(\rho(f,u,e),-)\eta_{\rho^f}^{-1},
		\end{equation}
		is a cross-connection such that it induces a \emph{dual cross-connection} $(\mathcal{L}_S,\mathcal{R}_S;\Delta_S)$ defined by the functor $\Delta_S\colon  \mathcal{L}_S \to N^*\mathcal{R}_S$ as follows:
		\begin{equation}
		v\Delta_S(Se) = H(\lambda^e;-) \quad\text{ and }\quad \Delta_S(\rho(e,u,f)) = \eta_{\lambda^e}\mathcal{R}_S(\lambda(f,u,e),-)\eta_{\lambda^f}^{-1}.
		\end{equation}
	\end{thm}
	
	This gives rise to the cross-connection semigroup
	$$\mathbb{S}\Gamma_S= (\mathcal{R}_S,\mathcal{L}_S;\Gamma_S)=\:\{\: (\rho^a,\lambda^a) : a\in S \} .$$
	
	Then the set of idempotents $E_{\Gamma_S}$ of the semigroup $\mathbb{S}\Gamma_S$ is given by the set:
	\begin{equation}\label{eqnegs}
	E_{\Gamma_S}=\{(Se,eS) : e\in E(S)\}.
	\end{equation}
	
	Observe that the element $(Se,eS)$ denotes the following pair of normal cones $(\gamma(Se,eS),\delta(Se,eS))= (\rho^e,\lambda^e) \in \mathbb{S}{\Gamma_S}$. Further if we define the partial orders $\omega^l$ and $\omega^r$ as follows:
	\begin{equation}\label{eqpo}
	(Se,eS)\omega^l(Sf,fS) \iff Se\subseteq Sf,\text{ and }(Se,eS)\omega^r(Sf,fS) \iff eS\subseteq fS,
	\end{equation}
	then $E_{\Gamma_S}$ forms a regular biordered set and it is biorder isomorphic to the biordered set $E$ of the semigroup $S$. If $(Se,eS),(Sf,fS)\in E_{\Gamma_S}$, then the transpose of a morphism $\rho(e,u,f)\in \mathcal{L}_S(Se,Sf) $ is the morphism $\lambda(f,u,e)\in \mathcal{R}_S(fS,eS)$. 
	
	Thus we obtain a cross-connection from a regular semigroup $S$. 
	
	\subsubsection{Regular semigroup of a cross-connection} Conversely, given a cross-connection $(\mathcal{D},\mathcal{C};{\Gamma})$ with the dual $\Delta$, by category isomorphisms \cite{mac}, we have two associated bifunctors $\Gamma(-,-)\colon  \mathcal{C}\times\mathcal{D} \to \bf{Set}$ and $\Delta(-,-)\colon  \mathcal{C}\times\mathcal{D} \to \bf{Set}$. There is a natural isomorphism $\chi_\Gamma$ between the bifunctors and $\chi_\Gamma$ is called the \emph{duality} associated with the cross-connection. Using the duality $\chi_\Gamma$, we can get a \emph{linking} of some normal cones in $T\mathcal{C}$ with those in $T\mathcal{D}$. The pairs of linked cones $(\gamma,\delta)$ will form a regular semigroup which is called the \emph{cross-connection semigroup} $\mathbb{S}\Gamma$ determined by $\Gamma$.
	$$ \mathbb{S}\Gamma = \:\{\: (\gamma,\delta) \in T\mathcal{C}\times T\mathcal{D} : (\gamma,\delta) \text{ is linked }\:\} $$ 
	For $(\gamma,\delta),( \gamma' , \delta') \in \mathbb{S}\Gamma$, the binary composition is defined by 
	$$ (\gamma , \delta) \: ( \gamma' , \delta') = (\gamma . \gamma' , \delta' . \delta)    .$$
	
	Using this correspondence between cross-connections and regular semigroups, Nambooripad \cite{cross} explicitly proved the equivalence of the category $\bf{Cr}$ of cross-connections and the category $\bf{RS}$ of regular semigroups.

	\section{Inductive groupoids from cross-connections}\label{ind}
	
	Recall that the overall aim of the article is to prove the equivalence between the inductive groupoid and the cross-connection of a regular semigroup. This shall be implemented by constructing the inductive groupoid directly from the cross-connection of a regular semigroup (in this section) and conversely building the cross-connection associated with the inductive groupoid of a regular semigroup (in the next section). 
	
	First, given the cross-connection $ \Gamma_S= (\mathcal{R}_S,\mathcal{L}_S;\Gamma_S)$ of a regular semigroup $S$ with biordered set $E$ and inductive groupoid $(\mathcal{G}(S),\epsilon_S)$, we proceed to construct an inductive groupoid $\mathcal{G}({\Gamma_S})$ associated with the cross-connection $\Gamma_S$. 
	
	For this end, we identify the regular biordered set associated with the cross-connection $\Gamma_S$, define the category $\mathcal{G}({\Gamma_S})$ as a suitable subcategory of the category $\mathcal{L}_S\times\mathcal{R}_S$ with identities as the regular biordered set and show that it forms a groupoid. Then we introduce a partial order $\leq_\Gamma$ in $\mathcal{G}({\Gamma_S})$ and show that $(\mathcal{G}({\Gamma_S}),\leq_\Gamma)$ is an ordered groupoid. After that, we shall define an evaluation functor $\epsilon_\Gamma$ from the groupoid $\mathcal{G}(E_{\Gamma_S})$ of $E$-chains to the ordered groupoid $\mathcal{G}({\Gamma_S})$ and prove that $\mathcal{G}({\Gamma_S})$ is inductive with respect to $\epsilon_\Gamma$. Finally we shall show that the inductive groupoid $(\mathcal{G}({\Gamma_S}),\epsilon_\Gamma)$ is inductive isomorphic to the inductive groupoid $(\mathcal{G}(S),\epsilon_S)$ of the semigroup $S$. 
	
	From the discussion in Section \ref{cxnrs}, it is clear that the regular biordered set associated with the cross-connection $\Gamma_S$ is the set $E_{\Gamma_S}$ as defined in (\ref{eqnegs}). So, the object set of our required groupoid $\mathcal{G}({\Gamma_S})$ is: 
	$$v\mathcal{G}({\Gamma_S})= v\mathcal{G}(E_{\Gamma_S}) = E_{\Gamma_S}=\{(Se,eS) : e\in E(S)\}.$$ 
	
	Given a morphism $(x,x') $ from $xx'=e$ to $x'x=f$ in $\mathcal{G}(S) $ such that $ x\in S$ and $ x'\in V(x)$; since $x\in R_e\cap L_f$, the mapping $\rho_x=\rho(e,x,f)$ is a morphism (in fact, an isomorphism) in  $\mathcal{L}_S$ from $Se$ to $Sf$. Similarly as $x'\in R_f\cap L_e$, the mapping $\lambda_{x'}=\lambda(e,x',f)$ is an isomorphism in the category $\mathcal{R}_S$ from $eS$ to $fS$. Thus, given two objects $(Se,eS), (Sf,fS) \in v\mathcal{G}({\Gamma_S})$, we define a morphism in $\mathcal{G}({\Gamma_S})$ from $(Se,eS)$ to $(Sf,fS)$ as a pair of isomorphisms $(\rho_x,\lambda_{x'})$ where $ x\in R_e\cap L_f$ and $ x'\in V(x)$ such that $x'\in R_f\cap L_e$. 
	
	Suppose $(\rho_x,\lambda_{x'})$ is a morphism from $(Se,eS)$ to $(Sf,fS)$, and $(\rho_y,\lambda_{y'})$ is a morphism from $(Sf,fS)$ to $(Sg,gS)$  in $\mathcal{G}({\Gamma_S})$. Then using the compositions in $\mathcal{L}_S$ and $\mathcal{R}_S$, we define the composition
	$$(\rho_x,\lambda_{x'})\ast (\rho_y,\lambda_{y'}) =(\rho_{xy},\lambda_{y'x'})$$ 
	so that $(\rho_{xy},\lambda_{y'x'})$ is a morphism from $(Se,eS)$ to $(Sg,gS)$.
	
	\begin{lem}
		$\mathcal{G}({\Gamma_S})$ is a groupoid.
	\end{lem}
	\begin{proof}
		First, to see that $\mathcal{G}({\Gamma_S})$ is a category, we need to verify associativity and identity. Given  $(\rho_x,\lambda_{x'})$ is a morphism from $(Se,eS)$ to $(Sf,fS)$, $(\rho_y,\lambda_{y'})$ is a morphism from $(Sf,fS)$ to $(Sg,gS)$ and $(\rho_z,\lambda_{z'})$ is a morphism from $(Sg,gS)$ to $(Sh,hS)$ in $\mathcal{G}({\Gamma_S})$. Then using the associativity of the compositions in $\mathcal{L}_S$ and $\mathcal{R}_S$, we can easily see that 
		$$(\rho_x,\lambda_{x'})\ast((\rho_y,\lambda_{y'})\ast (\rho_z,\lambda_{z'}))=(\rho_{xyz},\lambda_{z'y'x'})= ((\rho_x,\lambda_{x'})\ast(\rho_y,\lambda_{y'}))\ast(\rho_z,\lambda_{z'}).$$
		
		Also, given an object $(Se,eS)\in v\mathcal{G}({\Gamma_S})$, we have the identity morphism $(\rho_e,\lambda_e)=(\rho(e,e,e),\lambda(e,e,e))$ such that for a morphism $(\rho_x,\lambda_{x'})$ from $(Se,eS)$ to $(Sf,fS)$, we have $(\rho_e,\lambda_e)\ast(\rho_x,\lambda_{x'}) = (\rho_{ex},\lambda_{x'e}) =(\rho_x,\lambda_{x'})$. Similarly for a morphism $(\rho_y,\lambda_{y'})$ from $(Sg,gS)$ to $(Se,eS)$, we have $(\rho_y,\lambda_{y'})\ast(\rho_e,\lambda_e) = (\rho_{ye},\lambda_{ey'}) =(\rho_y,\lambda_{y'})$.
		
		Further, given a morphism $(\rho_x,\lambda_{x'})$ in $\mathcal{G}({\Gamma_S})$ from  $(Se,eS)$ to $(Sf,fS)$, we can see that $(\rho_{x'},\lambda_{x})$ is a morphism from $(Sf,fS)$ to $(Se,eS)$ such that  
		$$(\rho_x,\lambda_{x'})\ast (\rho_{x'},\lambda_{x}) =(\rho_{xx'},\lambda_{xx'})= (\rho_e,\lambda_{e})$$ 
		and
		$$(\rho_{x'},\lambda_{x})\ast (\rho_{x},\lambda_{x'}) =(\rho_{x'x},\lambda_{x'x})= (\rho_f,\lambda_{f}).$$ 
		Hence $\mathcal{G}({\Gamma_S})$ is a groupoid.
	\end{proof} 
	
	Now given a morphism $(\rho_x,\lambda_{x'})$ from $(Se,eS)$ to $(Sf,fS)$ and a morphism $(\rho_y,\lambda_{y'})$ from $(Sg,gS)$ to $(Sh,hS)$, define a relation $\leq_\Gamma$ on $\mathcal{G}({\Gamma_S})$ as follows: 
	\begin{equation}\label{po}
	(\rho_x,\lambda_{x'}) \leq_\Gamma (\rho_y,\lambda_{y'}) \iff 
	\begin{cases}
	(Se,eS) \subseteq (Sg,gS),\\
	(Sf,fS) \subseteq (Sh,hS), \text{ and }\\ 
	(\rho_x,\lambda_{x'})= (\rho_{ey}, \lambda_{y'e}).
	\end{cases}
	\end{equation}
	First, observe that the pair of isomorphisms $(\rho_{ey}, \lambda_{y'e})=(\rho(e,{ey},f),\lambda(e,{y'e},f))$ is a typical morphism in $\mathcal{G}({\Gamma_S})$ such that $e=eyy'e$ and $f=y'ey$.  
	
	\begin{lem}
		The relation $\leq_\Gamma$ is a partial order on $\mathcal{G}({\Gamma_S})$.
	\end{lem}
	\begin{proof}
		Since $x=ex$ and $x'=x'e$, clearly the relation $\leq_\Gamma$ is reflexive. 
		
		If $(\rho_x,\lambda_{x'}) \leq_\Gamma (\rho_y,\lambda_{y'})$ and $(\rho_y,\lambda_{y'}) \leq_\Gamma (\rho_x,\lambda_{x'})$, then $(Se,eS) \subseteq (Sg,gS)$, $(Sf,fS) \subseteq (Sh,hS)$, $(Sg,gS) \subseteq (Se,eS)$, $(Sh,hS) \subseteq (Sf,fS)$. So $(Se,eS) = (Sg,gS)$ and $(Sf,fS) = (Sh,hS)$. Also since $(\rho_x,\lambda_{x'})= (\rho_{ey}, \lambda_{y'e})$ and $(\rho_y,\lambda_{y'})= (\rho_{gx}, \lambda_{x'g})$,
		\begin{equation*}
		\begin{split}
		(\rho_x,\lambda_{x'})&=(\rho_{ey},\lambda_{y'e})\\
		&=(\rho_e\rho_y,\lambda_e\lambda_{y'})\\
		&=(\rho_e\rho_{gx},\lambda_e\lambda_{x'g})\\
		&=(\rho_{egx},\lambda_{x'ge})\\
		&=(\rho_{gx},\lambda_{x'g})\\
		&=(\rho_y,\lambda_{y'}).
		\end{split}
		\end{equation*}
		Hence $\leq_\Gamma$ is anti-symmetric.
		
		Now let $(\rho_x,\lambda_{x'}) \leq_\Gamma (\rho_y,\lambda_{y'})$ and $(\rho_y,\lambda_{y'}) \leq_\Gamma (\rho_z,\lambda_{z'})$ where $(\rho_z,\lambda_{z'})$ is a morphism from $(Sk,kS)$ to $(Sl,lS)$. Then $(Se,eS) \subseteq (Sg,gS)$, $(Sf,fS) \subseteq (Sh,hS)$, $(Sg,gS) \subseteq (Sk,kS)$, $(Sh,hS) \subseteq (Sl,lS)$. So $(Se,eS) \subseteq (Sk,kS)$, $(Sf,fS) \subseteq (Sl,lS)$. Also since $(\rho_x,\lambda_{x'})= (\rho_{ey}, \lambda_{y'e})$ and $(\rho_y,\lambda_{y'})= (\rho_{gz}, \lambda_{z'g})$,
		\begin{equation*}
		\begin{split}
		(\rho_x,\lambda_{x'})&=(\rho_{ey},\lambda_{y'e})\\
		&=(\rho_e\rho_y,\lambda_e\lambda_{y'})\\
		&=(\rho_e\rho_{gz},\lambda_e\lambda_{z'g})\\
		&=(\rho_{egz},\lambda_{z'ge})\\
		&=(\rho_{ez},\lambda_{z'e}).
		\end{split}
		\end{equation*}
		So $(\rho_x,\lambda_{x'}) \leq_\Gamma (\rho_z,\lambda_{z'})$, and thus $\leq_\Gamma$ is transitive.
		
		Hence the relation $\leq_\Gamma$ is a partial order.
	\end{proof}
	
	The above partial order may be discussed further in the context of \cite{kssnpo,hart,mitsch}. 
	Observe that $\leq_\Gamma$ on the identities of $\mathcal{G}({\Gamma_S})$ (or the biordered set $E_{\Gamma_S}$) reduces to the natural partial order $\omega$ on $E_{\Gamma_S}$, and may be written as follows. (Also see (\ref{eqpo}) above.) 
	$$(\rho_e,\lambda_{e}) \leq_\Gamma (\rho_f,\lambda_{f}) \iff (Se,eS)\subseteq(Sf,fS).$$
	
	Now define restriction and corestriction as follows. Given a morphism from $(\rho_x,\lambda_{x'})$ in $\mathcal{G}({\Gamma_S})$ from $(Se,eS)$ to $(Sf,fS)$, and if $(Sg,gS)\subseteq(Se,eS)$, then $\rho(g,g,e)$ is an inclusion in $\mathcal{L}_S$ from $Sg$ to $Se$, and $\lambda(g,g,e)$ is an inclusion in $\mathcal{R}_S$ from $gS$ to $eS$. So,
	$$(\rho(g,g,e)\rho(e,x,f),\lambda(g,g,e)\lambda(e,x',f))= (\rho(g,gx,f),\lambda(g,x'g,f))$$ 
	is a pair of monomorphisms in $\mathcal{L}_S \times \mathcal{R}_S$ from $(Sg,gS)$ to $(Sf,fS)$. Consider the pair of epimorphic components  $(\rho(g,gx,h),\lambda(g,x'g,h))$ where the idempotent $h= x'gx \in E(L_{gx})\cap E(R_{x'g})$. 
	
	So, the \emph{restriction} of $(\rho_x,\lambda_{x'})$ to $(Sg,gS)$ in $\mathcal{G}({\Gamma_S})$ is defined as the pair of isomorphisms $(\rho(g,gx,h),\lambda(g,x'g,h))= (\rho_{gx},\lambda_{x'g})$.
	
	Similarly, given a morphism from $(\rho_x,\lambda_{x'})$ in $\mathcal{G}({\Gamma_S})$ from $(Se,eS)$ to $(Sf,fS)$, and if $(Sh,hS)\subseteq(Sf,fS)$, then we have a pair of retractions $(\rho(f,h,h),\lambda(f,h,h))$ in $\mathcal{L}_S\times\mathcal{R}_S$ from $(Sf,fS)$ to $(Sh,hS)$. Then $$(\rho(e,x,f)\rho(f,h,h),\lambda(e,x',f)\lambda(f,h,h))= (\rho(e,xh,h),\lambda(e,hx',h)) $$ 
	is a pair of epimorphisms in $\mathcal{L}_S\times\mathcal{R}_S$. Consider their normal factorisations into retractions and isomorphisms as follows, 
	$$(\rho(e,xh,h),\lambda(e,hx',h)) = (\rho(e,g,g)\rho(g,xh,h),\lambda(e,g,g)\lambda(g,hx',h))$$ 
	where $g= xhx'$. So, we define the \emph{corestriction} of  $(\rho_x,\lambda_{x'})$ to $(Sh,hS)$ as the pair of isomorphisms $((\rho(g,xh,h),\lambda(g,hx',h))=(\rho_{xh},\lambda_{hx'})$.
	
	\begin{pro}
		$(\mathcal{G}({\Gamma_S}),\leq_\Gamma)$ is an ordered groupoid with restrictions and corestrictions defined as above.
	\end{pro}
	\begin{proof}
		We need to verify that $(\mathcal{G}({\Gamma_S}),\leq_\Gamma)$ satisfies the axioms of Definition \ref{og}. First, let $(\rho_x,\lambda_{x'})$ be a morphism from $(Se,eS)$ to $(Sf,fS)$ and $(\rho_y,\lambda_{y'})$ a morphism from $(Sf,fS)$ to $(Sg,gS)$ in $\mathcal{G}({\Gamma_S})$ so that $(\rho_{xy},\lambda_{y'x'})$ is a morphism from $(Se,eS)$ to $(Sg,gS)$. Also let $(\rho_u,\lambda_{u'})$ be a morphism in $\mathcal{G}({\Gamma_S})$ from $(Sh,hS)$ to $(Sk,kS)$ and $(\rho_v,\lambda_{v'})$ a morphism from $(Sk,kS)$ to $(Sl,lS)$ so that $(\rho_{uv},\lambda_{v'u'})$ is a morphism from $(Sh,hS)$ to $(Sl,lS)$. Suppose $(\rho_u,\lambda_{u'})\leq_\Gamma(\rho_x,\lambda_{x'})$ and $(\rho_v,\lambda_{v'})\leq_\Gamma(\rho_y,\lambda_{y'})$. Then clearly, $(Sh,hS) \subseteq (Se,eS)$ and $(Sl,lS) \subseteq (Sg,gS)$. Also since $(\rho_u,\lambda_{u'})= (\rho_{hx}, \lambda_{x'h})$ and  $(\rho_v,\lambda_{v'})= (\rho_{ky}, \lambda_{y'k})$, 
		\begin{equation*}
		\begin{split}
		(\rho_{uv},\lambda_{v'u'})&=(\rho_{u},\lambda_{u'})\ast(\rho_{v},\lambda_{v'})\\
		&=(\rho_{hx}, \lambda_{x'h}) \ast(\rho_{ky}, \lambda_{y'k})\\
		&=(\rho_{hxky},\lambda_{y'kx'h})\\
		&=(\rho_{hxy},\lambda_{y'x'h}).
		\end{split}
		\end{equation*}
		So $(\rho_{uv},\lambda_{v'u'})\leq_\Gamma(\rho_{xy},\lambda_{y'x'})$ and (OG1) is satisfied. 
		
		Second, let $(\rho_x,\lambda_{x'})$ be a morphism from $(Se,eS)$ to $(Sf,fS)$ and $(\rho_y,\lambda_{y'})$ a morphism from $(Sg,gS)$ to $(Sh,hS)$ in $\mathcal{G}({\Gamma_S})$. Then if $(\rho_x,\lambda_{x'}) \leq_\Gamma (\rho_y,\lambda_{y'})$, 
		then $(Se,eS) \subseteq (Sg,gS)$ and $(Sf,fS) \subseteq (Sh,hS)$. Also since $\rho_x=\rho_{ey}$, i.e., $\rho(e,x,f)=\rho(e,ey,f)$, from cross-connection theory (see Subsection \ref{cxnrs}), we have $x=ey=xx'y$. Similarly, $\lambda_{x'}=\lambda_{y'e}$ implies $ x'=y'e=y'xx'$. Then, 
		$$x= (xx')y=(yy'xx')y =yy'(xx'y) =yy'x=yy'(xx'x) = y(y'xx')x= yx'x=yf.$$ 
		This implies $\lambda(f,x,e)=\lambda(f,yf,e)$, i.e., $\lambda_x=\lambda_{yf}$. Similarly, $\lambda_{x'}=\lambda_{y'e}$ gives $\rho_{x'}=\rho_{fy'}$. Hence,
		$$(\rho_{x'},\lambda_{x})= (\rho_{fy'}, \lambda_{yf}).$$ 
		So $(\rho_{x'},\lambda_{x}) \leq_\Gamma (\rho_{y'},\lambda_{y})$. Hence (OG2) is satisfied.
		
		Lastly, for $(Sg,gS)\leq(Se,eS) $, if we define the {restriction} of $(\rho_x,\lambda_{x'})$ to $(Sg,gS)$ as the morphism $(\rho_{gx},\lambda_{x'g})$, then clearly (OG3) also holds. The dual condition (OG3$^*$) also holds for corestriction.
		
		Hence $(\mathcal{G}({\Gamma_S}),\leq_\Gamma)$ is an ordered groupoid.
	\end{proof}
	Now, since $E_{\Gamma_S}$ is a biordered set, $\mathcal{G}(E_{\Gamma_S})$ forms an ordered groupoid with respect to the order induced by the E-chains of $E_{\Gamma_S}$. In the sequel, for ease of notation we shall denote the vertex $(Se,eS)$ of the groupoid $\mathcal{G}(E_{\Gamma_S})$ by just $\mathfrak{e}$. Hence an E-chain $ ((Se_0,e_0S),(Se_1,e_1S),\dotsc,(Se_n,e_nS))$ shall be denoted by $ (\mathfrak{e}_0,\mathfrak{e}_1,\dotsc,\mathfrak{e}_n)$.
	
	We define a functor $\epsilon_\Gamma\colon\mathcal{G}(E_{\Gamma_S}) \to \mathcal{G}({\Gamma_S})$ as follows. We let 
	$v\epsilon_\Gamma =1_{E_{\Gamma_S}}$ and for an arbitrary E-chain $\mathfrak{c}= (\mathfrak{e}_0,\mathfrak{e}_1,\dotsc,\mathfrak{e}_n)$  in $\mathcal{G}(E_{\Gamma_S})$, 
	$$\epsilon_\Gamma(\mathfrak{c}) = (\rho_{w},\lambda_{w'})$$
	where $w=e_0e_1\dots e_{n-1}e_n$ and $w'=e_ne_{n-1}\dots e_1e_0$.
	
	First, $w\in e_0Se_n$ and $w'\in e_nSe_0$, so $(\rho(e_0,w,e_n),\lambda(e_0,w',e_n))\in \mathcal{L}_S\times\mathcal{R}_S$.
	
	Observe that either $e_{i-1}\mathscr{R}e_i$ so that $e_{i-1}e_i=e_i$, or $e_{i-1}\mathscr{L}e_i$ so that $e_{i-1}e_i=e_{i-1}$. So, $ww'=e_0e_1\dots e_{n-1}e_ne_ne_{n-1}\dots e_1e_0=e_0$ and similarly $w'w=e_n$. Since $ww'w=w$ and $w'ww'=w'$, we have $w'\in V(w)$. Hence $  (\rho_{w},\lambda_{w'})$ is a morphism in the groupoid $\mathcal{G}({\Gamma_S})$.
	
	\begin{lem}
		The functor	$\epsilon_\Gamma \colon \mathcal{G}(E_{\Gamma_S}) \to \mathcal{G}({\Gamma_S})$ is a $v$-isomorphism.
	\end{lem}
	\begin{proof}
		First we need to verify that $\epsilon_\Gamma $ is a well-defined functor. Let $\mathfrak{c},\mathfrak{d} \in \mathcal{G}(E_{\Gamma_S})$ where $\mathfrak{c}$ is an $E$-chain as defined in the foregoing discussion and  $\mathfrak{d}= (\mathfrak{f}_0,\mathfrak{f}_1,\dotsc,\mathfrak{f}_n)$ where $v=f_0f_1\dots f_{n-1}f_n$ and $v'=f_nf_{n-1}\dots f_1f_0$. Suppose $\mathfrak{c}.\mathfrak{d}$ exists so that $\epsilon_\Gamma(\mathfrak{c})$ and $\epsilon_\Gamma(\mathfrak{d})$ are composable. Then,
		\begin{equation*}
		\begin{split}
		\epsilon_\Gamma(\mathfrak{c})\ast\epsilon_\Gamma(\mathfrak{d}) &= (\rho_{w},\lambda_{w'})\ast(\rho_{v},\lambda_{v'})\\
		&=(\rho_w\rho_v,\lambda_{w'}\lambda_{v'})\\
		&=(\rho_{wv},\lambda_{v'w'})\\
		&=(\rho_{(wv)},\lambda_{(wv)'})\\
		&=\epsilon_\Gamma(\mathfrak{c}.\mathfrak{d}).
		\end{split}
		\end{equation*}
		
		Then for an $E$-chain $\mathfrak{c}= (\mathfrak{e}_0,\mathfrak{e}_1,\dotsc,\mathfrak{e}_n) \in \mathcal{G}(E_{\Gamma_S})$ and $\mathfrak{h}\in \omega(\mathfrak{e}_0)$, let
		$$\mathfrak{h}\cdot \mathfrak{c}=  (\mathfrak{h},\mathfrak{h}_1,\dotsc,\mathfrak{h}_n)=  ((Sh,hS),(Sh_1,h_1S),\dotsc,(Sh_n,h_nS))$$
		where $h_i=e_ih_{i-1}e_i$ for all $i=1,\dots, n$ and $h_0=h$.
		
		Observe that,
		$$hh_1\dots h_n=h(e_1e_0he_0e_1)\dots(e_ne_{n-1}\dots e_1e_0he_0e_1\dots e_{n-1}e_n)=he_0e_1\dots e_{n-1}e_n.$$
		Similarly, $h_nh_{n-1}\dots h_1h= e_ne_{n-1}\dots e_1e_0h.$
		
		So, $$\epsilon_\Gamma(\mathfrak{h}{\downharpoonleft} \mathfrak{c})= (\rho_{he_0e_1\dots e_{n-1}e_n},\lambda_{e_ne_{n-1}\dots e_1e_0h}).$$
		
		Also, 
		\begin{equation*}
		\begin{split}
		\mathfrak{h}{\downharpoonleft} \epsilon_\Gamma(\mathfrak{c})&= \mathfrak{h}{\downharpoonleft} (\rho_{e_0e_1\dots e_{n-1}e_n},\lambda_{e_ne_{n-1}\dots e_1e_0})\\
		&=(\rho_{he_0e_1\dots e_{n-1}e_n},\lambda_{e_ne_{n-1}\dots e_1e_0h}).
		\end{split}
		\end{equation*}
		Hence $\epsilon_\Gamma(\mathfrak{h}{\downharpoonleft} \mathfrak{c}) =\mathfrak{h}{\downharpoonleft} \epsilon_\Gamma(\mathfrak{c})$, and $\epsilon_\Gamma$  is a $v$-isomorphism.
	\end{proof}
	
	\begin{thm}
		$(\mathcal{G}({\Gamma_S}),\epsilon_\Gamma)$ is an inductive groupoid.
	\end{thm}
	\begin{proof}
		We need to verify that $(\mathcal{G}({\Gamma_S}),\epsilon_\Gamma)$ satisfies the axioms of Definition \ref{dfnig}. 
		
		First, we need to prove (IG1) of Definition \ref{dfnig}. That is, for an inductive groupoid $(G,\epsilon_G)$, let $x\in \mathcal{G}$ and for $i=1,2$, given $e_i$, $f_i \in E$ such that $\epsilon (e_i) \leq \mathbf{d}(x)$ and $\epsilon (f_i) = \mathbf{r}(\epsilon (e_i){\downharpoonleft} x)$. If $e_1\omega^r e_2$, then $f_1\omega^r f_2$, and
		$$\epsilon (e_1,e_1e_2)(\epsilon (e_1e_2){\downharpoonleft} x) = (\epsilon (e_1){\downharpoonleft} x)\epsilon (f_1,f_1f_2).$$
		
		Now, let $(\rho_x,\lambda_{x'})$ be a morphism in $\mathcal{G}({\Gamma_S})$ from $(Se,eS)$ to $(Sf,fS)$ such that for $i=1,2$, $(S{e_i},{e_i}S) \subseteq (Se,eS)$. Then 		$\mathbf{r}((\rho_{e_ix},\lambda_{x'e_i})) $ is $(S{f_i},{f_i}S)= (S{x'e_ix},{x'e_ix}S)$. 
		If $e_1S\subseteq e_2S$, then 
		$f_2f_1 = x'e_2xx'e_1x=x'e_2ee_1x=x'e_1x=f_1.$
		So $f_1S\subseteq f_2S$. Then using the notations of the previous lemma, we need to show that:
		$$\epsilon_\Gamma(\mathfrak{e}_1,\mathfrak{e}_1\mathfrak{e}_2)\ast\epsilon_\Gamma(\mathfrak{e}_1\mathfrak{e}_2){\downharpoonleft} (\rho_x,\lambda_{x'}) = \epsilon_\Gamma(\mathfrak{e}_1){\downharpoonleft}(\rho_x,\lambda_{x'})\ast\epsilon_\Gamma(\mathfrak{f}_1,\mathfrak{f}_1\mathfrak{f}_2).	$$
		For that end, observe that:
		\begin{equation*}
		\begin{split}
		\epsilon_\Gamma(\mathfrak{e}_1,\mathfrak{e}_1\mathfrak{e}_2)\ast\epsilon_\Gamma(\mathfrak{e}_1\mathfrak{e}_2){\downharpoonleft} (\rho_x,\lambda_{x'}) 
		&= (\rho_{e_1e_2},\lambda_{e_1e_2e_1})\ast (\rho_{e_1e_2x},\lambda_{x'e_1e_2})\\
		&= (\rho_{e_1e_2},\lambda_{e_1})\ast (\rho_{e_1e_2x},\lambda_{x'e_1e_2})\\
		&= (\rho_{e_1e_2x},\lambda_{x'e_1e_2e_1})\\
		&= (\rho_{e_1e_2x},\lambda_{x'e_1}).\\
		\end{split}
		\end{equation*}
		Also,
		\begin{equation*}
		\begin{split} \epsilon_\Gamma(\mathfrak{e}_1){\downharpoonleft}(\rho_x,\lambda_{x'})\ast\epsilon_\Gamma(\mathfrak{f}_1,\mathfrak{f}_1\mathfrak{f}_2)&= (\rho_{e_1x},\lambda_{x'e_1})\ast(\rho_{f_1f_2},\lambda_{f_1f_2f_1})\\
		&= (\rho_{e_1x},\lambda_{x'e_1})\ast(\rho_{f_1f_2},\lambda_{f_1})\\
		&= (\rho_{e_1x(x'e_1x)(x'e_2x)},\lambda_{x'e_1x)x'e_1})\\
		&= (\rho_{e_1(xx')e_1(xx')e_2x)},\lambda_{x'e_1(xx')e_1})\\
		&= (\rho_{e_1e_2x},\lambda_{x'e_1}).\\
		\end{split}
		\end{equation*}
		Thus we have verified (IG1) and similarly we can verify its dual.
		
		Now we have to verify (IG2) that every singular $E$-square in the ordered groupoid $\mathcal{G}(E_{\Gamma_S})$ is $\epsilon_\Gamma$-commutative. Let $\bigl[ \begin{smallmatrix} \mathfrak{g}&\mathfrak{g}\mathfrak{e}\\ \mathfrak{h}&\mathfrak{h}\mathfrak{e} \end{smallmatrix} \bigr]$
		be a column-singular E-square such that $g,h \in \omega^r(e)$ and $g\mathscr{L} h$. Then,
		\begin{equation*}
		\begin{split}
		\epsilon_\Gamma(\mathfrak{g},\mathfrak{h})\ast\epsilon_\Gamma(\mathfrak{h},\mathfrak{he})&=(\rho_{gh},\lambda_{hg})\ast(\rho_{he},\lambda_{heh})\\
		&=(\rho_{ghe},\lambda_{hehg})\\
		&=(\rho_{ghe},\lambda_{hg})\\
		&=(\rho_{ge},\lambda_{h}).
		\end{split}
		\end{equation*}
		Also,
		\begin{equation*}
		\begin{split}
		\epsilon_\Gamma(\mathfrak{g},\mathfrak{ge})\ast\epsilon_\Gamma(\mathfrak{ge},\mathfrak{he})&=(\rho_{ge},\lambda_{geg})\ast(\rho_{gehe},\lambda_{hege})\\
		&=(\rho_{ghe},\lambda_{hegeg})\\
		&=(\rho_{ge},\lambda_{hg})\\
		&=(\rho_{ge},\lambda_{h}).
		\end{split}
		\end{equation*}
		So a column-singular E-square is $\epsilon_\Gamma$-commutative. Dually, we can show that a row-singular E-square is also $\epsilon_\Gamma$-commutative. So (IG2) also holds.
		
		Hence $(\mathcal{G}({\Gamma_S}),\epsilon_\Gamma)$ is an inductive groupoid.
	\end{proof}
	Observe that the above proof is an adaptation of the proof of $(\mathcal{G}(S),\epsilon_S)$ being an inductive groupoid \cite[Theorem 3.8]{mem}.
	\begin{thm}
		$(\mathcal{G}({\Gamma_S}),\epsilon_\Gamma)$ is inductive isomorphic to $(\mathcal{G}(S),\epsilon_S)$.
	\end{thm}
	\begin{proof}
		Define a functor $\Phi\colon \mathcal{G}({\Gamma_S})$ to $\mathcal{G}(S)$ as follows:
		$$v\Phi(\mathfrak{e}) = v\Phi ((Se,eS)) = e \text{ and } \Phi((\rho_x,\lambda_{x'}))= (x,x').$$
		Clearly $\Phi$ is a covariant functor. Recall that $v\mathcal{G}({\Gamma_S})= v\mathcal{G}(E_{\Gamma_S}) = E_{\Gamma_S}$ and $E_{\Gamma_S}$ is biorder isomorphic to the biordered set $E$ of $S$ under the mapping $v\Phi\colon(Se,eS)\mapsto e$.
		
		We can see that $\Phi$ is order preserving since
		\begin{equation*}
		\begin{split}
		\Phi((Se,eS){\downharpoonleft}(\rho_x,\lambda_{x'}))&= \Phi((\rho_{ex},\lambda_{x'e}))\\
		&= ({ex},{x'e})\\
		&= e{\downharpoonleft}({x},{x'})\\
		&=\Phi(Se,eS){\downharpoonleft}\Phi(\rho_x,\lambda_{x'}),
		\end{split}
		\end{equation*}
		
		Also given an $E$-chain $(\mathfrak{e}_0,\mathfrak{e}_1,\dots,\mathfrak{e}_n)\in  \mathcal{G}(E_{\Gamma_S}) $,
		\begin{equation*}
		\begin{split}
		\mathcal{G}(v\Phi)\epsilon_S(\mathfrak{e}_0,\mathfrak{e}_1,\dots,\mathfrak{e}_n)&= \epsilon_S( e_0,e_1,\dots,e_n)\\
		&=(e_0e_1\dots e_n,e_ne_{n-1}\dots e_0)\\
		&= \Phi(\rho_{e_0e_1\dots e_n},\lambda_{e_ne_{n-1}\dots e_0})\\
		&= \epsilon_\Gamma \Phi (\mathfrak{e}_0,\mathfrak{e}_1,\dots,\mathfrak{e}_n).
		\end{split}
		\end{equation*}
		Hence the following diagram commutes:
		\begin{equation*}\label{indfun}
		\xymatrixcolsep{4pc}\xymatrixrowsep{3pc}\xymatrix
		{
			\mathcal{G}(E_{\Gamma_S}) \ar[r]^{\mathcal{G}(v\Phi)} \ar[d]_{\epsilon_\Gamma}  
			& \mathcal{G}(E) \ar[d]^{\epsilon_S} \\       
			\mathcal{G}({\Gamma_S}) \ar[r]^{\Phi} & \mathcal{G}(S) 
		}
		\end{equation*}
		Moreover, by the definition of the morphism $(\rho_x,\lambda_{x'})$ in $\mathcal{G}({\Gamma_S})$, it is clear that $\Phi$ is full and faithful. Hence $\Phi$ is an inductive isomorphism between $(\mathcal{G}({\Gamma_S}),\epsilon_\Gamma)$ and $(\mathcal{G}(S),\epsilon_S)$.
	\end{proof}
	
	\section{Cross-connections from inductive groupoids}\label{cxn}
	
	Having built `the' inductive groupoid $(\mathcal{G}(S),\epsilon_S)$ of the semigroup $S$ from the cross-connection ${\Gamma_S}=(\mathcal{R}_S,\mathcal{L}_S;{\Gamma}_S)$ in the previous section, now we attempt the converse. Given the inductive groupoid $\mathcal{G}(S)$ with biordered set $E$, we proceed to construct a cross-connection $(\mathcal{R}_G,\mathcal{L}_G;{\Gamma_G})$ and show that it is cross-connection isomorphic to ${\Gamma}_S$.
	
	On contrary to the previous section where the inductive groupoid was suitably identified from the cross-connection, here we will have to `extract' the suitable parts from the inductive groupoid $\mathcal{G}(S)$ and combine to form the required cross-connected categories $\mathcal{L}_G$ and $\mathcal{R}_G$. This shall be achieved as follows. We first build three categories from $\mathcal{G}(S)$: namely $\mathcal{P}_L$, $\mathcal{Q}_L$ and $\mathcal{G}_L$. Then these categories shall be suitably `glued' together to form the first normal category $\mathcal{L}_G$. Similarly the second category $\mathcal{R}_G$ shall be built from the categories $\mathcal{P}_R$, $\mathcal{Q}_R$ and $\mathcal{G}_R$, all extracted from the inductive groupoid $\mathcal{G}(S)$. Then we shall define a cross-connection ${\Gamma_G}$ between $\mathcal{R}_G$ and $\mathcal{L}_G$, and finally prove that it is `the' cross-connection associated the inductive groupoid. This plan is illustrated in Figure \ref{figindcxn}.
	
	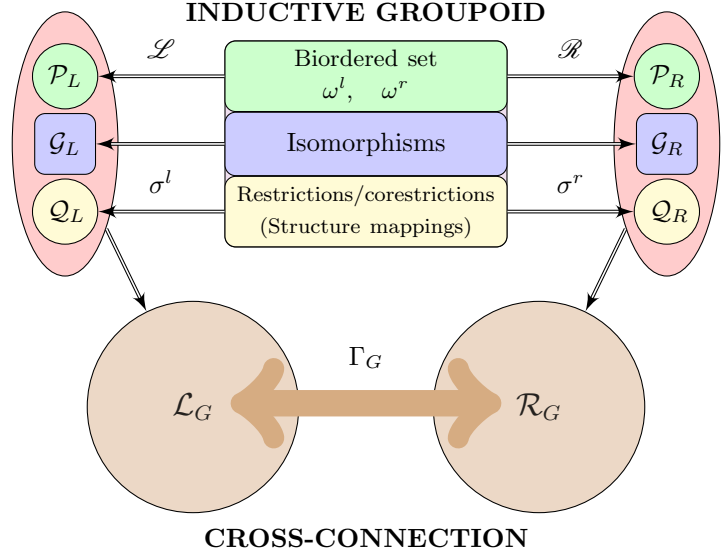
\begin{figure}
		\begin{center}
			\begin{tikzpicture}[align=center,node distance=2cm]
			{\node[rectangle, draw, fill=violet!25, fill opacity=.5, text width=10em, text centered, rounded corners, minimum height=7.8em,minimum width=5em] (ig) at (0,4){};
				\node[minimum height=.5cm, minimum width=.6cm,anchor=south] at (0,5.5) {\bf INDUCTIVE GROUPOID};
				
				\node[rectangle, draw, fill=green!20, text width=10em, text centered, rounded corners, minimum height=2.5em,minimum width=5em] (bs) at (0,4.9){\small Biordered set \\ $\omega^l, \quad \omega^r$}; 
				\node[rectangle, draw, fill=blue!20,text width=10em, text centered, rounded corners, minimum height=2.4em,minimum width=5em] (igg) at (0,4){Isomorphisms}; 
				\node[rectangle, draw, fill=yellow!20, text width=10em, text centered, rounded corners, minimum height=2.4em,minimum width=5em] (sm) at (0,3.1){\footnotesize Restrictions/corestrictions (Structure mappings)}; 
			}
			{
				\node[ellipse, draw, fill=red!20, text width=1.3em, text centered,minimum height=10em,minimum height=10em,minimum width=4em,] (lq) at (-4,4){}; 
			}
			{
				\node[circle, draw, fill=green!20, text width=1.3em, text centered] (pl) at (-4,4.9){ $\mathcal{P}_L$}; 
				\draw[-latex',double ] (bs) -- node[label=90:{ $\mathscr{L}$}] {}  (pl);
			}
			{
				\node[rectangle, draw, fill=blue!20, text width=1.2em, text centered, rounded corners, minimum height=2.3em,minimum width=2.3em] (gl) at (-4,4){ $\mathcal{G}_L$}; 
				\draw[-latex',double ] (igg) -- node[label=90:{}] {}  (gl);
			}
			{
				\node[circle, draw, fill=yellow!20, text width=1.3em, text centered] (ql) at (-4,3.1){ $\mathcal{Q}_L$}; 
				\draw[-latex',double ] (sm) -- node[label=90:{$\sigma^l$}] {}  (ql);
			}
			{
				\node[circle, draw, fill=brown!35, fill opacity=.85, text width=2em, text centered, minimum height=8em,minimum width=8em] (lg) at (-2.3,0.5){\Large $\mathcal{L}_G$}; 
				\draw[-latex',double ] (lq) -- node[label=240:{}] {}  (lg);
			}
			{
				\node[ellipse, draw, fill=red!20, text width=1.3em, text centered,minimum height=10em,minimum height=10em,minimum width=4em,] (rq) at (4,4){}; 
				\node[circle, draw, fill=green!20, text width=1.3em, text centered] (pr) at (4,4.9){ $\mathcal{P}_R$}; 
				\node[rectangle, draw, fill=blue!20, text width=1.2em, text centered, rounded corners, minimum height=2.3em,minimum width=2.3em] (gr) at (4,4){ $\mathcal{G}_R$}; 
				\node[circle, draw, fill=yellow!20, text width=1.3em, text centered] (qr) at (4,3.1){ $\mathcal{Q}_R$}; 
				
				\draw[-latex',double ] (bs) -- node[label=90:{$\mathscr{R}$}] {}  (pr);
				\draw[-latex',double ] (igg) -- node[label=90:{}] {}  (gr);
				\draw[-latex',double ] (sm) -- node[label=90:{$\sigma^r$}] {}  (qr);
			}
			
			{
				\node[circle, draw, fill=brown!35,fill opacity=.85, text width=2em, text centered,minimum height=8em,minimum width=8em] (rg) at (2.3,0.5){\Large $\mathcal{R}_G$}; 
				\draw[-latex',double ] (rq) -- node[label=360:{}] {}  (rg);
			}
			{
				\draw[brown!65,<->,line width=10pt] (-1.8,.55) --  (1.8,.55);
				\node[minimum height=.7cm, minimum width=.7cm,anchor=south] at (0,.8) {$\Gamma_G$};
				\node[minimum height=.5cm, minimum width=.6cm,anchor=south] at (0,-1.5) {\bf CROSS-CONNECTION};
			}
			\end{tikzpicture}
		\end{center}
		\caption{Cross-connection of an inductive groupoid}\label{figindcxn}
	\end{figure}
	
	First, we proceed to build the category $\mathcal{L}_G$. Given the inductive groupoid $\mathcal{G}(S)$ with regular biordered set $E$, let $v\mathcal{L}_G= E/\mathscr{L}$. This gives a partially ordered set $E/\mathscr{L}$ with respect to the order $\omega^l/\mathscr{L} =\leq_L$. In fact, $E/\mathscr{L}$ forms a regular partially ordered set, in the sense of Grillet \cite{gril}. The proof of this statement may be found in \cite{bicxn}. Given $e\in E$, in the sequel $\overleftarrow{e}$ shall denote the canonical image of $e$ in $E/\mathscr{L}$. 
	
	Now to define morphisms on $\mathcal{L}_G$, we define three categories --- $\mathcal{P}_L$, $\mathcal{Q}_L$ and $\mathcal{G}_L$ such that $v\mathcal{L}_G= v\mathcal{P}_L= v\mathcal{Q}_L= v\mathcal{G}_L=E/\mathscr{L}$. It must be mentioned here that these categories were earlier considered by Rajan \cite{rajancat,rajancat1} as subcategories derived from normal categories. Here we are building them from inductive groupoids. 
	
	First, if $\overleftarrow{e}\leq_L \overleftarrow{f}$, we define as a morphism in $\mathcal{P}_L$, a unique morphism $j_\mathcal{L}(e,f)= j_\mathcal{L}(e,e,f)$ from $\overleftarrow{e}$ to $\overleftarrow{f}$. So, given two morphisms $j_\mathcal{L}(e,f)$ and $j_\mathcal{L}(g,h)$, they are equal if and only if $e\mathscr{L}g$ and $f\mathscr{L}h$. Given $j_\mathcal{L}(e,f)$ and $j_\mathcal{L}(f,g)$, we compose them using the binary composition induced by the partial binary composition of the biordered set $E$ as follows:
	$$j_\mathcal{L}(e,f)j_\mathcal{L}(f,g) = j_\mathcal{L}(e,g).$$
	
	Observe that since $e\omega^l f$, we have $e f = e$ in $E$. Now, we show that the partially ordered set $v\mathcal{P}_L$ can be realised as a category $\mathcal{P}_L$. In fact, exactly this idea is the cornerstone of Nambooripad's generalisation of Grillet's construction.
	
	\begin{lem}
		$\mathcal{P}_L$ is a strict preorder category with the object set $v\mathcal{P}_L = E/\mathscr{L}$ and the morphisms in $\mathcal{P}_L$ as defined above. 
	\end{lem}
	\begin{proof}
		We need to verify associativity and identity. Given morphisms $j_\mathcal{L}(e,f)$, $j_\mathcal{L}(f,g)$ and $j_\mathcal{L}(g,h)$, since
		$$(j_\mathcal{L}(e,f)\: j_\mathcal{L}(f,g))\: j_\mathcal{L}(g,h) = j_\mathcal{L}(e,h) = j_\mathcal{L}(e,f)\: (j_\mathcal{L}(f,g)\: j_\mathcal{L}(g,h)),$$
		associativity holds. Also since, 
		$$j_\mathcal{L}(e,e)\: j_\mathcal{L}(e,f) = j_\mathcal{L}(e,f)$$ and $$j_\mathcal{L}(e,f)\: j_\mathcal{L}(f,f) = j_\mathcal{L}(e,f),$$
		$j_\mathcal{L}(e,e)$ is the identity morphism at $\overleftarrow{e}$. Also observe that by the defnition, there is exactly a unique morphism $j_\mathcal{L}(e,f)$ between any two objects $\overleftarrow{e}$ and $\overleftarrow{f}$ in $v\mathcal{P}_L$. Hence $\mathcal{P}_L$ is a strict preorder category.
	\end{proof}
	
	Now, recall from \cite{grillet} the definition of structure mappings of a regular semigroup. If $f \omega^l e$, then the structure mapping $\psi^{e}_{f}\colon L_{e} \to L_{f}$ is defined as $\psi^{e}_{f}\colon  x \mapsto xf$. Dually, if $f \omega^r e$, the structure mapping $\phi^{e}_{f}\colon R_{e} \to R_{f}$ is defined as $\phi^{e}_{f}\colon  x \mapsto fx$.
	
	So, if $\overleftarrow{e}\leq_L \overleftarrow{f}$, for each idempotent $e\in \overleftarrow{e}$, we define a morphism in $\mathcal{Q}_L$ from $\overleftarrow{f}$ to $\overleftarrow{e}$ as $q_\mathcal{L}(f,fe,e)$ using the structure mapping $\psi^{f}_{e}$. Equivalently, for each $u\in E(L_{e})\cap\omega(f)$, we have a morphism $q_\mathcal{L}(f,u,e)$ from $\overleftarrow{f}$ to $\overleftarrow{e}$. It may be verified that these morphisms also constitute a category $\mathcal{Q}_L$ such that the composition is defined by semigroup composition. 
	
	Recall that the structure mappings carry the same information as the restrictions/corestrictions, which in turn capture the sandwich sets. Thus, the morphisms in $\mathcal{Q}_L$ inculcate the notion of the sandwich sets; this fact shall be elaborated later. But it must be noticed that different structure mappings may give rise to the same retraction and so the correspondence is not one-one.
	
	Further, given a morphism $(x,x') $ in the inductive groupoid $ \mathcal{G}(S)$ between $e$ and $f$, we define a morphism in $\mathcal{G}_L$ as the morphism $g_\mathcal{L}(e,x,f)$ from $\overleftarrow{e}$ and $\overleftarrow{f}$. Again, we can show that $\mathcal{G}_L$ forms a groupoid such that $(g_\mathcal{L}(e,x,f))^{-1}= g_\mathcal{L}(f,x',e)$ where $x'\in R_{f}\cap L_{e}$. (Also see \cite{rajancat2} for another independent approach by Rajan to build the groupoid $\mathcal{G}_L$.)
	
	Now, to construct the normal category $\mathcal{L}_G$ from the categories $\mathcal{P}_L$, $\mathcal{Q}_L$ and $\mathcal{G}_L$, we need the following concept of a quiver.
	
	A \emph{quiver} $\mathcal{Q}$ consists of a set of objects (denoted as $v\mathcal{Q}$) together with a set of morphisms (denoted by $\mathcal{Q}$ itself) and two functions $\mathbf{d},\mathbf{r}\colon\mathcal{Q}\rightrightarrows v\mathcal{Q}$ giving the \emph{domain} and \emph{codomain} of each morphism. Morphisms $f,g\in\mathcal{Q}$ are said to be \emph{composable} if $\mathbf{r}(f)=\mathbf{d}(g)$.
	The \emph{free category} $\overline{\mathcal{Q}}$ generated by a quiver $\mathcal{Q}$ is the category with $v\overline{\mathcal{Q}} =v\mathcal{Q}$ and with the following morphisms:
	\begin{enumerate}
		\item the identity morphisms at each object in $v\mathcal{Q}$;
		\item the morphisms in $\mathcal{Q}$;
		\item the sequences $f_1f_2\dots f_n$ of morphisms in $\mathcal{Q}$ such that $f_i$ and $f_{i+1}$ are composable for each $i=1,2,\dots,n-1$.
	\end{enumerate}
	For an illustration, the Figure \ref{figqui} shows (on the left) a quiver with three objects $a,b,c$ and three morphisms $f$, $g$ and $h$  and (on the right) the free category generated by this quiver.
	\begin{figure}
		\centering
		\begin{tikzpicture}[->,>=stealth',shorten >=1pt,auto,node distance=3cm,
		thick,main node/.style={circle,draw,font=\sffamily\Large\bfseries}]
		
		\node[state]         (0)         {$b$};
		\node[state]         (1) [below right of=0] {$c$};
		\node[state]         (2) [below left of=0] {$a$};
		
		\draw[transform canvas={yshift=0.5ex,xshift=-0.5ex},->] (2) --node{$f$} (0);
		\draw[transform canvas={yshift=-0.5ex,xshift=0.5ex},->] (2) --node[below]{$g$} (0);
		\path (0) edge   node {$h$} (1);

		\node                (7) [right of=0] {};
		
		\node[state]         (3) [right of=7]       {$b$};
		\node[state]         (4) [below right of=3] {$c$};
		\node[state]         (5) [below left of=3] {$a$};
		
		\path (3) edge   node {$h$} (4)
		
		(3) edge    [loop above]      node {$1_b$} (3)
		
		(4) edge    [loop above]      node {$1_c$} (4)
		
		(5) edge    [loop above]       node {$1_a$} (5);
		
		\draw[transform canvas={yshift=0.5ex,xshift=-0.5ex},->] (5) --node{$f$} (3);
		\draw[transform canvas={yshift=-0.5ex,xshift=0.5ex},->] (5) --node[below]{$g$} (3);
		\draw[transform canvas={yshift=0.75ex},->] (5) --node{$fh$} (4);        
		\draw[transform canvas={yshift=-0.75ex},->] (5) --node[below]{$gh$} (4);   
		
		\end{tikzpicture}
		\caption{The free catgeory generated by a quiver}\label{figqui}
	\end{figure}
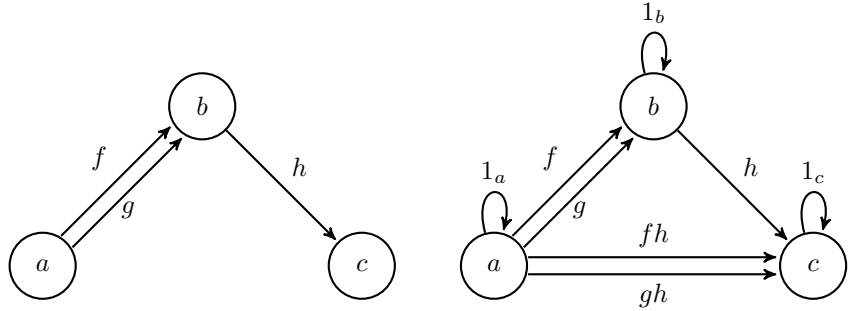
	
	
	Proceeding with our construction, we define $\mathcal{L}_Q$ with object set $v\mathcal{L}_Q= E/\mathscr{L}$ and morphisms as follows:
	$$\mathcal{L}_Q=\mathcal{P}_L \cup \mathcal{Q}_L \cup \mathcal{G}_L.$$
	
	Then clearly $\mathcal{L}_Q$ forms a quiver. Observe that given an arbitrary morphism $r=r(e,u,f)$ in $\mathcal{L}_Q$, the morphism $r$ belongs to the category $\mathcal{P}_L$ if $u=e$, the morphism $r$ belongs to $\mathcal{Q}_L $ if $u\in E(L_{f})\cap\omega(e)$ and the morphism $r$ belongs to the category $ \mathcal{G}_L$ if $u \in R_e\cap L_f$. 
	
	Now let $r_1=r(e,u,f)$ and $r_2=r(g,v,h)$ denote any two morphisms in the quiver $\mathcal{L}_Q$ from $\overleftarrow{e}$ to $\overleftarrow{f}$ and $\overleftarrow{g}$ to $\overleftarrow{h}$ respectively. Then the morphisms $r_1$ and $r_2$ of $\mathcal{L}_Q$ are said to be composable if $f\mathscr{L}g$. Also, we define a composition using the semigroup composition as follows:
	$$r(e,u,f)\: r(g,v,h) = r(e,uv,h).$$
	
	Let $\overline{\mathcal{L}_Q}$ be the free category generated by the quiver $\mathcal{L}_Q$ under the above composition. Observe that $u\in eSf$ and $v\in fSg$ for any morphism in $\mathcal{L}_Q$ and so we can see that $uv=x\in eSg$ for any morphism $r(e,x,g)$ in $\overline{\mathcal{L}_Q}$. Conversely, given any element $x\in eSf$, we can see that 
	\begin{equation}\label{normf}
	r(e,x,f) = r(e,g,g)r(g,x,h)r(h,h,f)
	\end{equation}
	where $h\in E(L_x)$ and $g\in E(R_x)\cap \omega(e)\neq\emptyset$ (since $S$ is regular). Then $r(e,g,g)\in \mathcal{Q}_L$, $r(g,x,h)\in \mathcal{G}_L$ and $r(h,h,g) \in \mathcal{P}_L$; hence $r(e,x,f)\in \overline{\mathcal{L}_Q}$.
	Further given any two morphisms (by abuse of notation) $r(e,u,f)$ and $r(g,v,h)$ in $\overline{\mathcal{L}_Q}$, we define a relation $\sim$ between the morphisms as 
	$$r(e,u,f)\sim r(g,v,h) \iff e\mathscr{L} g, f\mathscr{L} h, u=ev.$$ 
	\begin{lem}
		$\sim$ is an equivalence relation on the set $\overline{\mathcal{L}_Q}$.
	\end{lem}
	\begin{proof}
		Clearly $\sim$ is reflexive. 
		
		Suppose  $r(e,u,f)\sim r(g,v,h)$, then $e\mathscr{L} g$, $f\mathscr{L} h$ and  $u=ev$. Since, 
		$$v = gv = (ge)v = g(ev) = gu,$$
		so $r(g,v,h) \sim r(e,u,f)$. Hence $\sim$ is symmetric.
		
		Also if $r(e,u,f)\sim r(g,v,h)$ and $r(g,v,h)\sim r(k,w,l)$, then $e\mathscr{L} g\mathscr{L} k$, $f\mathscr{L}h \mathscr{L}  l$, $u=ev$ and $v=gw$.
		Then $$u=ev=e(gw)= (eg)w= ew.$$
		
		So $r(e,u,f)\sim  r(k,w,l)$ and $\sim$ is transitive.
		Hence $\sim$ is an equivalence relation.
	\end{proof}
	
	Now we are in a position to define the first normal category $\mathcal{L}_G$ of the required cross-connection as the quotient of the category $\overline{\mathcal{L}_Q}$ over $\sim$. Recall that the set of of objects $v\mathcal{L}_G$ is the partially ordered set $E/\mathscr{L}$. A morphism in the category $\mathcal{L}_G$ is defined as the $\sim$-class of $r(e,u,f)$ in $\overline{\mathcal{L}_Q}$. In the sequel, again by abuse of notation, we shall denote the $\sim$-class of $r(e,u,f)$ as a morphism in $\mathcal{L}_G$ by $r(e,u,f)$ itself. 
	
	We have already seen that any element of $eSf$ gives rise to a morphism. Also observe that if $r(e,u,f)=r(e,v,f)$ in $\mathcal{L}_G$, for $u,v\in eSf$, then since $u= ev= v$, the correspondence is one-one. Hence we have a bijection between the morphisms between $\overleftarrow{e}$ and $\overleftarrow{f}$ in $\mathcal{L}_G$ with the set $eSf$.
	
	Now we proceed to show that $\mathcal{L}_G$ forms a category under the same composition as in $\mathcal{L}_Q$: 
	$$r(e,u,f)\: r(f,v,g) = r(e,uv,g).$$
	
	\begin{lem}
		$\mathcal{L}_G$ is a category.
	\end{lem}
	\begin{proof}
		We need to verify associativity and identity. Associativity follows from the associativity of the semigroup multipication. Given a morphism $r(e,u,f)$ in $\mathcal{L}_G$,
		$$r(e,u,f)r(f,f,f)=r(e,u,f)\text{ and }r(e,e,e)r(e,u,f)=r(e,u,f).$$ 
		So, $\rho(e,e,e)$ is the identity morphism at $\overleftarrow{e}$. Hence $\mathcal{L}_G$ is a category.
	\end{proof}
	
	Observe that the category $\mathcal{P}_L$ is a subcategory of $\mathcal{L}_G$. We can also realise $\mathcal{Q}_L$ and $\mathcal{G}_L$ (in fact, the $\sim$-images of these categories) as subcategories of $\mathcal{L}_G$.
	
	\begin{lem}\label{lemcso}
		$(\mathcal{L}_G,\mathcal{P}_L)$ is a category with subobjects.
	\end{lem}
	\begin{proof}
		Clearly, $\mathcal{L}_G$ is a small category and $\mathcal{P}_L$ is a strict preorder subcategory of $\mathcal{L}_G$ such that $v\mathcal{L}_G = v\mathcal{P}_L$. Let $r(e,e,f)$ be a morphism in $\mathcal{P}_L$. Then if $r(g,u,e)r(e,e,f)=r(h,v,e)r(e,e,f)$, then $r(g,ue,f)=r(h,ve,f)$. So $g\mathscr{L}h$, $e\mathscr{L}e$ and  
		$$u=ue=gve=(gv)e=ue.$$
		Hence $r(g,u,e)=r(h,v,e)$, and so $r(e,e,f)\in \mathcal{P}_L$ is a monomorphism.
		
		Now if $r(e,e,f)=r(g,u,h)r(h,h,k)$ for $r(e,e,f),r(h,h,k)\in \mathcal{P}_L$ and $r(g,u,h) \in \mathcal{L}_G$, then 
		since $r(g,u,h)r(h,h,k)=r(g,uh,k)$, we have $r(e,e,f) = r(g,uh,k)$. So, $e\mathscr{L}g$ and $f\mathscr{L}k$. Since
		$u=uh=ge=g$, we have $r(g,u,h)=r(g,g,h)$, i.e., $r(g,u,h)\in \mathcal{P}_L$.
		Hence $(\mathcal{L}_G,\mathcal{P}_L)$ is a category with subobjects.
	\end{proof}
	
	So, a morphism $r(e,e,f)=j_\mathcal{L}(e,f) \in \mathcal{P}_L$ is an \emph{inclusion} in $\mathcal{L}_G$. Given an inclusion $j_\mathcal{L}(e,f)$, we can see that
	$$r(e,e,f)r(f,fe,e)= r(e,efe,e)=r(e,e,e).$$
	Hence, we have the following lemma.
	\begin{lem}\label{lemsplit}
		Every inclusion in $\mathcal{L}_G$ \emph{splits}.
	\end{lem}
	Also observe that a right inverse $r(f,fe,e)$ of an inclusion is a morphism in $\mathcal{Q}_L$. We can show that, in fact, every morphism in $\mathcal{Q}_L$ is a right inverse of an inclusion. Hence every morphism in $\mathcal{Q}_L$ is a \emph{retraction} in $\mathcal{L}_G$.
	
	Also, for a morphism $r(e,x,f)\in \mathcal{G}_L\subseteq \mathcal{L}_G$, there is a morphism $r(f,x',e)\in \mathcal{G}_L$ such that 
	$$r(e,x,f)r(f,x',e)=r(e,xx',e)=r(e,e,e)$$
	and 
	$$r(f,x',e)r(e,x,f)=r(f,x'x,f)=r(f,f,f).$$ 
	So a morphism in $\mathcal{G}_L$ is an \emph{isomorphism} in the category $\mathcal{L}_G$.
	
	Hence, given any morphism $r=r(e,x,f)$ in $\mathcal{L}_G$, by (\ref{normf}), it has a \emph{normal factorization}. Then the morphism $r^\circ = r(e,gx,h)$ is the \emph{epimorphic component} of the morphism $r$.
	
	Now, we proceed to discuss the normal cones in the category $\mathcal{L}_G$. First compare Definition \ref{dfnnm} and Definition \ref{dfnnc} to see that a normal cone is the exact categorical generalisation of a normal mapping. We shall now construct some important normal cones in $\mathcal{L}_G$ which arise from the inductive structure of $S$. 
	
	Let $(x,x')$ be a morphism from $e$ to $f$ in $\mathcal{G}(S)$. Observe that $f=x'x\in E(L_x)$. Consider an arbitrary idempotent $g\in E$. Then since $S$ is regular, the sandwich set $\mathcal{S}(g,e)$ is not empty. Let $h\in \mathcal{S}(g,e)$ so that $h\omega^re$. Then $he\omega e$ and we have the restriction of $(x,x')$ to $he$, i.e., $he{\downharpoonleft}(x,x')=(hex,x'he)$ which arises from the structure mapping $\phi^e_h$. Similarly, we have a structure mapping $\psi^g_h$ which give rise to a retraction $r(g,gh,h)$ from $\overleftarrow{g}$ to $\overleftarrow{h}=\overleftarrow{gh}$. 
	
	Then the morphism $r(h,he,he)$ is an isomorphism from $\overleftarrow{h}$ to $\overleftarrow{he}$ which arises from the evaluation of the E-chain $(h,he)\in \mathcal{G}(E)$. The morphism $r(he,hex,k)$ which is also an isomorphism from $\overleftarrow{he}$ to $\overleftarrow{k}$, where $k= x'hex$, arises from the restriction $(hex,x'he)$ of the morphism $(x,x')$ to $he$. Observe that we are only using the `left part' of the morphism $(hex,x'he)$. Also since $k\omega f$, we have an inclusion $r(k,k,f)$ from $\overleftarrow{k}$ to $\overleftarrow{f}$. This can be illustrated using Figure \ref{figncone}.
	
	\begin{figure}
		\centering
		\begin{equation*}
		\xymatrixcolsep{4pc}\xymatrixrowsep{3pc}\xymatrix
		{
			&e \ar@{-->}@/^.7pc/[r]^{(x,x')} \ar@{.}[rdd]_{\phi^e_h}^{\omega}&f\\  
			g\ar@{->}[ddr]_{\omega}^{\psi^g_h}\\
			&h \ar[r]^{\mathscr{R}}&he\ar@{.}[d]\ar@/^.7pc/[r]^{(hex,x'he)} &k\ar[luu]_{\omega}\\      
			&gh \ar@{.}[u]_{\mathscr{L}}\ar@{.}[r] & ge 
		}
		\end{equation*}
		\caption{Normal cone arising from a morphism in the inductive groupoid}\label{figncone}
	\end{figure}
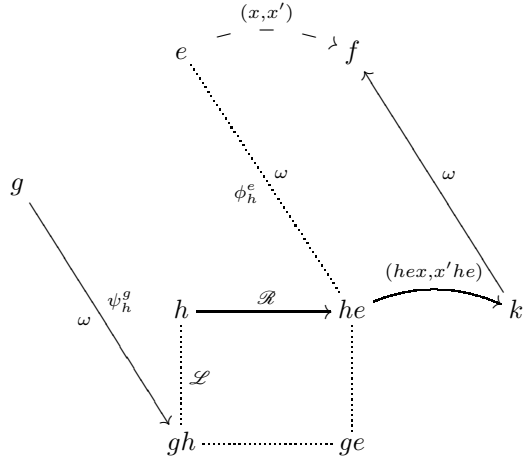
	
	In Figure \ref{figncone}, the dashed arrow represents the morphism $(x,x')\in\mathcal{G}(S)$. The solid arrows give rise to the relevant morphisms in $\mathcal{L}_G$ arising from the inductive structure of $S$. 
	
	Composing these morphisms, we get a morphism from $\overleftarrow{g}$ to $\overleftarrow{f}$ in $\mathcal{L}_G$ as given below:
	\begin{equation*}
	\begin{split}
	r(g,gh,h)r(h,he,he)r(he,hex,k)r(k,k,f)&=r(g,ghehexk,f)\\
	&=r(g,ghe(hexk),f)\\
	&=r(g,g(hehe)x,f)\\
	&=r(g,(ghe)x,f)\\
	&=r(g,(ge)x,f)\\
	&=r(g,gx,f).
	\end{split}
	\end{equation*}
	Observe that the morphism is independent of the choice of the sandwich element. Similarly we can build morphisms from every $\overleftarrow{g}\in v\mathcal{L}_G$, and so we define a map $r^x\colon v\mathcal{L}_G \to \mathcal{L}_G $ as follows:
	$$r^x(\overleftarrow{g})=r(g,gx,f).$$
	Clearly $r^x$ satisfies (Ncone1). If $\overleftarrow{g}' \leq_L \overleftarrow{g}$, then $g'\omega^lg$ and so we have an inclusion $r(g',g',g)$. Then 
	$$r(g',g',g)r^x(\overleftarrow{g})=r(g',g',g)r(g,gx,f) =r(g',g'gx,f)= r(g',g'x,f).$$ 
	Also since  $r^x(\overleftarrow{g}')=r(g',g'x,f)$, (Ncone2) is also satisfied.
	
	Since $ex=x$, we see that $r^x(\overleftarrow{e})=r(e,ex,f)=r(e,x,f)$ which is nothing but the isomorphism in $\mathcal{L}_G$ arising from the morphism $(x,x')\in \mathcal{G}(S)$. Hence we have an isomorphism component for $r^x$ and (Ncone3) is also satisfied. 
	
	Thus $r^x$ is a normal cone in $\mathcal{L}_G$ with apex $\overleftarrow{f}$. This normal cone shall be called, by abuse of terminology, as the principal cone induced by the element $x\in S$ in the category $\mathcal{L}_G$. (Also see Section \ref{cxnrs}.) Observe how the sandwich set $\mathcal{S}(g,e)$ is `partially' encoded into the normal cone $r^x$ through the retraction $r(g,gh,h)$ and the isomorphism $r(h,hex,k)$. Also observe that these morphisms arose via the structure mappings $\psi^g_h$ and $\phi^e_h$. 
	
	As a result of these inbuilt sandwich elements in the normal cones, unlike morphisms in inductive groupoids, actually it is relatively straightforward to compose principal cones. Given two principal cones $r^a$ and $r^b$ with apices $\overleftarrow{e}$ and $\overleftarrow{f}$ respectively, we get a new principal cone $r^a\star r^b$ as follows:
	\begin{equation}\label{compprinc}
	r^a\star r^b (\overleftarrow{g}) = r^a(\overleftarrow{g}) (r^b(\overleftarrow{e}))^\circ.
	\end{equation}  
	Here $(r^b(\overleftarrow{e}))^\circ$ represents the epimorphic component of the morphism $r^b(\overleftarrow{e})$. Observe that since $aeb=ab$, for $h\in E(L_{eb})=E(L_{ab})$,
	$$r^a\star r^b (\overleftarrow{g})=r(g,ga,e)r(e,eb,h) = r(g,gaeb,h)=r(g,gab,h)=r^{ab} (\overleftarrow{g}).$$
	Hence if $e\in E$ is an idempotent, $r^e$ is an idempotent normal cone.
	
	Now, we are in a position to prove the normality of the category $\mathcal{L}_G$.
	\begin{thm}
		$(\mathcal{L}_G,\mathcal{P}_L)$ forms a normal category.
	\end{thm}
	\begin{proof}
		By Lemma \ref{lemcso}, $(\mathcal{L}_G,\mathcal{P}_L)$ is a category with subobjects. As already seen, any morphism in $\mathcal{L}_G$ has a normal factorization as described in (\ref{normf}) and hence (NC1) is satisfied. By Lemma \ref{lemsplit} above, every inclusion in $\mathcal{L}_G$ splits, thus satisfying (NC2). Moreover, for an object $\overleftarrow{e}\in\mathcal{L}_G$, we see that $r^e$ is an idempotent normal cone with apex $\overleftarrow{e}$. Hence (NC3) is also holds. Thus $(\mathcal{L}_G,\mathcal{P}_L)$ is a normal category.
	\end{proof}
	
	Recall that Grillet constructed an intermediary regular semigroup $N(X)$ of all normal mappings from a regular partially ordered set $X$. This regular semigroup may be seen as a generic regular semigroup arising from the regular partially ordered set $X$. One can see that Nambooripad's semigroup $T\mathcal{C}$ of all normal cones arising from a normal category $\mathcal{C}$ is an exact generalisation of this, where the semigroup composition (see (\ref{eqnsg1}) in Section \ref{cxns}) arises from the composition of principal cones as given above in (\ref{compprinc}). It must be admitted here that the `genericness' of the semigroup $T\mathcal{C}$ is yet to be properly formulated and established. An interested reader can find some discussion on this, focussing on some concrete special classes in \cite{sunny,var,tlx,tx}.
	
	Having built the normal category $\mathcal{L}_G$ from the inductive groupoid $\mathcal{G}(S)$, dually, we can build a normal category $(\mathcal{R}_G,\mathcal{P}_R)$ as follows. Let $v\mathcal{R}_G=v\mathcal{P}_R= E/\mathscr{R}$. This gives a regular partially ordered set with respect to $\omega^r/\mathscr{R} =\leq_R$. We shall denote by $\overrightarrow{e}$, the canonical image of $e$ in $E/\mathscr{R}$.
	
	Now if $\overrightarrow{e}\leq_R \overrightarrow{f}$, then we define the unique morphism $j_\mathcal{R}(e,f)$ from $\overrightarrow{e}$ to $\overrightarrow{f}$ as a morphism in $\mathcal{P}_R$. This collection of morphisms forms a strict preorder category $\mathcal{P}_R$ with the object set $E/\mathscr{R}$ under the following composition: 
	$$j_\mathcal{R}(e,f)\: j_\mathcal{R}(f,g)= j_\mathcal{R}(e,g)$$
	Observe that since $e\omega^rf$, $f\: e=e \in E$. 
	
	Then we define the category $\mathcal{Q}_R$ with object set $v\mathcal{Q}_R=E/\mathscr{R}$ as follows. If $\overrightarrow{e} \leq_R \overrightarrow{f}$, for each $e\in \overrightarrow{e}$, $q_\mathcal{R}(f,ef,e)$ from $\overrightarrow{f}$ to $\overrightarrow{e}$ using the structure mapping $\phi^f_e$, i.e., for each $h\in E(R_{e})\cap\omega(f)$, we have a morphism $\lambda(f,h,e)$ from $\overrightarrow{f}$ to $\overrightarrow{e}$.
	
	Further, given a morphism $(x,x') \in \mathcal{G}(S)$ from $e$ to $f$, we define a category $\mathcal{G}_R$ with $v\mathcal{G}_R=E/\mathscr{R}$ and a morphism $g_\mathcal{R}(e,x',f)$ from $\overrightarrow{e}$ to $\overrightarrow{f}$ in $\mathcal{G}_R$.
	
	Let $\mathcal{R}_Q$ be a quiver with $v\mathcal{R}_Q= E/\mathscr{R}$ and morphisms  
	$$\mathcal{R}_Q=\mathcal{P}_R\cup\mathcal{Q}_R\cup\mathcal{G}_R.$$ 
	If $l(e,u,f)$ and $l(f,v,g)$ denote any two morphisms in $\mathcal{R}_Q$ from $\overrightarrow{e}$ to $\overrightarrow{f}$ and $\overrightarrow{f}$ to $\overrightarrow{g}$ respectively, then we define a composition as follows:
	$$l(e,u,f)\: l(f,v,g) = r(e,vu,g).$$
	
	We define $\overline{\mathcal{R}_Q}$ as the free category generated by the quiver $\mathcal{R}_Q$ under the above composition. Observe that for any morphism $l(e,x,f)$ in $\overline{\mathcal{R}_Q}$, $x\in fSe$ and conversely, given any element $x\in fSe$, we can see that
	\begin{equation}\label{normfr}
	l(e,x,f) = l(e,g,g)l(g,x,h)l(h,h,f)
	\end{equation}
	where $h\in E(R_x)$ and $g\in E(L_x)\cap \omega(e)$. Then $l(e,g,g)\in \mathcal{Q}_R$, $l(g,x,h)\in \mathcal{G}_R$ and $l(h,h,g) \in \mathcal{P}_R$; hence $l(e,x,f)\in \overline{\mathcal{R}_Q}$.
	
	Further we can define an equivalence relation $\sim$ for any two morphisms (by abuse of notation) $l(e,u,f)$ and $l(g,v,h)$ in $\overline{\mathcal{R}_Q}$ as 
	$$l(e,u,f)\sim l(g,v,h) \iff e\mathscr{R} g, f\mathscr{R} h, u=ve.$$ 
	Now we define the $\sim$-class of a morphism $l(e,u,f)$ in $\overline{\mathcal{R}_Q}$ as a morphism in the category $\mathcal{R}_G$. We shall denote a morphism in $\mathcal{R}_G$ by $l(e,u,f)$ itself. Then $\mathcal{R}_G$ forms a category with $v\mathcal{R}_G=E/\mathscr{R}$. This leads to the following dual theorem:
	\begin{thm}
		$(\mathcal{R}_G,\mathcal{P}_R)$ is a normal category such that a principal cone $l^x$ is given by, for each $\overrightarrow{g}\in v\mathcal{R}_G$,
		$$l^x(\overrightarrow{g})=l(g,xg,f) \text{ where }f\in E(R_x).$$ 
	\end{thm}
	
	Having constructed two normal categories $\mathcal{L}_G$ and $\mathcal{R}_G$, now we need to define a cross-connection $\Gamma_G$ between these categories, and then establish the equivalence. For convenience, we take an alternate route. We first show that the categories $\mathcal{L}_G$ and $\mathcal{R}_G$ are normal category isomorphic to the categories $\mathcal{L}_S$ and $\mathcal{R}_S$ of the semigroup $S$, respectively. Then we shall use Nambooripad's construction of ${\Gamma_S}$ in \cite{cross} using $\mathcal{L}_S$ and $\mathcal{R}_S$ of the semigroup $S$, to build our cross-connection $\Gamma_G$. This is justified since the construction is lengthy and the exact imitation of the construction suffices, once we establish normal category isomorphisms.
	
	So, define a functor $\mathfrak{L} \colon \mathcal{L}_G \to \mathcal{L}_S$ as follows:
	\begin{equation}\label{l}
	v\mathfrak{L} (\overleftarrow{e})=Se\text{ and } \mathfrak{L} (r(e,u,f)) = \rho(e,u,f).
	\end{equation}
	\begin{pro}
		$\mathfrak{L} $ is a normal category isomorphism.
	\end{pro}
	\begin{proof}
		First, observe $\overleftarrow{e}=\overleftarrow{f}$ if and only if $e\mathscr{L}f$ if and only if $Se=Sf$. Hence $v\mathfrak{L} $ is well-defined.
		
		Since $r(e,u,f)=r(g,v,h)$ if and only if $e\mathscr{L}g$, $f\mathscr{L}h$, $u\in eSf$, $v\in gSh$ and $u=ev$, if and only if $\rho(e,u,f)=\rho(g,v,h)$ (by Section \ref{cxnrs}); $\mathfrak{L} $ is well-defined. 
		
		\begin{equation*}
		\begin{split}
		\mathfrak{L} (r(e,u,f)r(f,v,g))&=\mathfrak{L} (r(e,uv,g))\\
		&=\rho(e,uv,g)\\
		&=\rho(e,u,f)\rho(f,v,g)\\
		&=\mathfrak{L} (r(e,u,f)\mathfrak{L} (r(f,v,g)).
		\end{split}
		\end{equation*}
		Also since $\mathfrak{L} (r(e,e,e))=\rho(e,e,e)$, $\lambda$ is functor.
		
		Given an inclusion $r(e,e,f)\in\mathcal{L}_G$ from $\overleftarrow{e}$ to $\overleftarrow{f}$, then $\mathfrak{L} (r(e,e,f))=\rho(e,e,f)$ is an inclusion in $\mathcal{L}_S$ from $Se$ to $Sf$. So $\mathfrak{L} $ is inclusion preserving.
		
		The functor $\mathfrak{L} $ is clearly full and $v$-full. Retracing the steps in the proof of well-definedness of $\mathfrak{L} $ shows that $\mathfrak{L} $ is an order isomorphism, $v$-injective and faithful. Hence $\mathcal{L}_G$ and $\mathcal{L}_S$ are isomorphic as normal categories.
	\end{proof}
	
	Dually, $\mathcal{R}_G$ is isomorphic to $\mathcal{R}_S$ via the functor $\mathfrak{R} \colon \mathcal{R}_G \to \mathcal{R}_S$ defined as follows:
	\begin{equation}\label{r}
	v\mathfrak{R} (\overrightarrow{e})=eS\text{ and } \mathfrak{R} (l(e,u,f)) = \lambda(e,u,f).
	\end{equation}
	
	Now, we proceed to construct the required cross-connection $\Gamma_G$. Define a functor $\Gamma_G\colon  \mathcal{R}_G \to N^*\mathcal{L}_G$ as follows:
	\begin{equation}
	v\Gamma_G(\overrightarrow{e}) = H(r^e;-) \text{ and } \Gamma_G(l(e,u,f)) = \eta_{r^e}\mathcal{L}_G(r(f,u,e),-)\eta_{r^f}^{-1}
	\end{equation}
	where $r^e(\overleftarrow{g})=r(g,ge,e) \in \mathcal{L}_G$ and $\eta_{r^e}$ is the natural isomorphism between the $H$-functor $H(r^e;-)$ and the hom-functor $\mathcal{L}_G(\overleftarrow{e},-)$. The $H$-functor $H(r^e;-)\in vN^*\mathcal{L}_G$ is such that $H(r^e;-)\colon \mathcal{L}_G\to \mathbf{Set}$ is determined by the principal cone $r^e\in T\mathcal{L}_G$. The covariant hom-functor $\mathcal{L}_G(\overleftarrow{e},-)\colon \mathcal{L}_G\to \mathbf{Set}$ is the hom-functor determined by the object $\overleftarrow{e}\in v\mathcal{L}_G$. The natural transformation $\Gamma_G(l(e,u,f))$ may be described by the following commutative diagram:
	\begin{equation*}\label{eta}
	\xymatrixcolsep{2pc}\xymatrixrowsep{4pc}\xymatrix
	{
		H(r^e;-) \ar[rr]^{\eta_{r^e}} \ar[d]_{\Gamma_G(l(e,u,f))}  
		&& \mathcal{L}_G(\overleftarrow{e},-) \ar[d]_{\mathcal{L}(r(f,u,e),-)} &\overleftarrow{e}\\       
		H(r^f;-) \ar[rr]^{\eta_{r^f}} && \mathcal{L}_G(\overleftarrow{f},-)&\overleftarrow{f}\ar[u]_{r(f,u,e)} 
	}
	\end{equation*}
	
	Now, we proceed to show that $\Gamma_G$ is a local isomorphism. But for that end, we shall need the following theorem which relates the normal category $\mathcal{R}_{T\mathcal{C}}$ of principal right ideals of the regular semigroup ${T\mathcal{C}}$, with the normal dual $N^\ast\mathcal{C}$ of a normal category $\mathcal{C}$.
	\begin{thm}\cite[Theorem III.25]{cross}\label{nd}
		The category $\mathcal{R}_{T\mathcal{C}}$ is normal category isomorphic to the normal dual $N^\ast\mathcal{C}$. 
	\end{thm}
	\begin{pro}
		The functor $\Gamma_G$ is a local isomorphism.
	\end{pro}
	\begin{proof}
		We shall prove this proposition by appealing to the proof of \cite[Proposition IV.1]{cross} and \cite[Theorem IV.2]{cross} which shows that $\Gamma_S$ is a local isomorphism. 
		
		First, since $\mathcal{L}_G$ is a normal category, by Theorem \ref{nd} the normal category $\mathcal{R}_{T\mathcal{L}_G}$ arising from the principal right ideals of $T\mathcal{L}_G$ is isomorphic to the normal dual $N^\ast\mathcal{L}_G$, say, via the functor $\bar{G}$. So, all we have to show in order to prove the proposition is that there is a local isomorphism from $\mathcal{R}_{G}$ to $\mathcal{R}_{T\mathcal{L}_G}$. Define a functor 
		$\bar{F}\colon \mathcal{R}_{G} \to \mathcal{R}_{T\mathcal{L}_G}$ as follows:
		\begin{equation*}
		v\bar{F}(\overrightarrow{e}) = r^e(T\mathcal{L}_G) \quad\text{ and }\quad \Gamma_G(l(e,u,f)) = \lambda(r^e,r^u,r^f).
		\end{equation*}
		Imitating the proof of \cite[Proposition IV.1]{cross}, we can show that $\bar{F}$ (denoted by $FS_\rho$ in \cite{cross}) is a local isomorphism. Hence $\bar{F}\bar{G}\colon \mathcal{R}_G \to N^*\mathcal{L}_G$ is a local isomorphism. Since $\Gamma_G=\bar{F}\bar{G}$ (read $\Gamma_S=FS_\rho \circ\bar{G}$ in \cite{cross}), the functor $\Gamma_G$ is a local isomorphism.
	\end{proof}
	\begin{thm}
		$\Gamma_G$ is a cross-connection.
	\end{thm}
	\begin{proof}
		By the above proposition, $\Gamma_G$ is a local isomorphism. Also, since $\overleftarrow{e}\in Mr^e$, for every $\overleftarrow{e} \in v\mathcal{L}_G$, we have $\overrightarrow{e} \in v\mathcal{R}_G$ such that $\overleftarrow{e} \in M\Gamma_G(\overrightarrow{e})=Mr^e$. Hence $\Gamma_G$ is a cross-connection.
	\end{proof}
	
	Dually, we can show that $(\mathcal{L}_G,\mathcal{R}_G;\Delta_G)$ defined by the functor $\Delta_G\colon  \mathcal{L}_G \to N^*\mathcal{R}_G$ as follows is a cross-connection.
	\begin{equation}
	v\Delta_G(\overleftarrow{e}) = H(l^e;-) \text{ and }\Delta_G(r(e,u,f)) = \eta_{l^e}\mathcal{R}_G(l(f,u,e),-)\eta_{l^f}^{-1}.
	\end{equation}
	As in the case of ${\Gamma_S}$ in \cite{cross}, $\Gamma_G$ gives rise to the cross-connection semigroup
	$$\mathbb{S}\Gamma_G= (\mathcal{R}_G,\mathcal{L}_G;\Gamma_G)=\:\{\: (r^a,l^a) : a\in S \} $$
	so that the set of idempotents $E_{\Gamma_G}$ of the semigroup $\mathbb{S}\Gamma_G$ is given by the set
	$$E_{\Gamma_G}=\{(\overleftarrow{e},\overrightarrow{e}) : e\in E(S)\}.$$
	Here the element $(\overleftarrow{e},\overrightarrow{e})$ denotes the pair of normal cones $(\gamma(\overleftarrow{e},\overrightarrow{e}),\delta(\overleftarrow{e},\overrightarrow{e}))= (r^e,l^e) \in \mathbb{S}{\Gamma_G}$. Further if we define the partial orders $\omega^l$ and $\omega^r$ as follows:
	\begin{equation}\label{eqpog}
	(\overleftarrow{e},\overrightarrow{e})\omega^l(\overleftarrow{f},\overrightarrow{f}) \iff \overleftarrow{e}\leq_L\overleftarrow{f},\text{ and } (\overleftarrow{e},\overrightarrow{e})\omega^r(\overleftarrow{f},\overrightarrow{f}) \iff \overrightarrow{e}\leq_R\overrightarrow{f},
	\end{equation}
	then $E_{\Gamma_G}$ forms a regular biordered set and it is biorder isomorphic to the biordered set $E$ of the inductive groupoid $\mathcal{G}(S)$ (also see \cite{bicxn}).
	
	Also, we can verify that if $(\overleftarrow{e},\overrightarrow{e}),(\overleftarrow{f},\overrightarrow{f})\in E_{\Gamma_G}$, then the transpose of $r(e,u,f)\in \mathcal{L}_G(\overleftarrow{e},\overleftarrow{f})$ is the morphism $l(f,u,e)\in \mathcal{R}_G(\overrightarrow{f},\overrightarrow{e})$. 
	
	Now, we need to show that the cross-connection $(\mathcal{R}_{G},\mathcal{L}_{G},\Gamma_{G})$ of the inductive groupoid $\mathcal{G}(S)$ is cross-connection isomorphic to cross-connection $(\mathcal{R}_S,\mathcal{L}_S,\Gamma_S)$ of the semigroup $S$. 
	\begin{thm}
		$(\mathcal{R}_{G},\mathcal{L}_{G},\Gamma_{G})$ is isomorphic to $(\mathcal{R}_S,\mathcal{L}_S,\Gamma_S)$ as cross-connections.
	\end{thm}
	\begin{proof}
		We already have seen that there are normal category isomorphisms between the corresponding categories via the functors $\mathfrak{L} $ and $\mathfrak{R} $ defined by (\ref{l}) and (\ref{r}) respectively. So all we need to verify is $m=(\mathfrak{L} ,\mathfrak{R} )$ is a cross-connection morphism as in Definition \ref{mor}.
		
		First, if $(\overleftarrow{e},\overrightarrow{e})\in E_{\Gamma_G}$, then $(\mathfrak{L} (\overleftarrow{e}),\mathfrak{R} (\overrightarrow{e}))= (Se,eS)\in E_{\Gamma_S}$. Also for any $\overleftarrow{g}\in v\mathcal{L}_G$, 
		\begin{equation*}
		\begin{split}
		\mathfrak{L} (r^e(\overleftarrow{g}))&= \mathfrak{L} (r(e,eg,g))\\
		&= \rho(e,eg,g)\\
		&= \rho^e(Sg)\\
		&=\gamma(Se,eS)(Sg)\\
		&=\gamma(\mathfrak{L} (\overleftarrow{e}),\mathfrak{R} (\overrightarrow{e}))(\mathfrak{L} (\overleftarrow{g})).
		\end{split}
		\end{equation*}
		Hence (M1) is satisfied.
		
		Also if $(\overleftarrow{e},\overrightarrow{e}),(\overleftarrow{f},\overrightarrow{f})\in E_{\Gamma_G}$, then the transpose $r(e,u,f)^\ast$ of $r(e,u,f)\in \mathcal{L}_G(\overleftarrow{e},\overleftarrow{f})$ is the morphism $l(f,u,e)\in \mathcal{R}_G(\overrightarrow{f},\overrightarrow{e})$. Then, 
		\begin{equation*}
		\begin{split}
		\mathfrak{R} (l(f,u,e))&= \lambda(f,u,e)\\
		&=\rho(e,u,f)^\ast\\
		&=(\mathfrak{L} (r(e,u,f)))^\ast.
		\end{split}
		\end{equation*}
		So (M2) is also satisfied and hence $(\mathcal{R}_{G},\mathcal{L}_{G},\Gamma_{G})$ is cross-connection isomorphic to $(\mathcal{R}_S,\mathcal{L}_S,\Gamma_S)$.
	\end{proof}
	
	\section{Summary and future directions}\label{sum}
	As mentioned earlier, in a follow up article \cite{indcxn2}, we shall `lift' the equivalence established in this article to abstract inductive groupoids and cross-connections, thus establishing the equivalence of these constructions completely independent of semigroups. This in turn suggests that any discussion on inductive groupoids is equivalent to one on cross-connection theory, and vice versa.
	
	Comparing these constructions, a regular biordered set may be seen equivalent to a pair of cross-connected regular partially ordered sets, as shown earlier by Nambooripad \cite{bicxn}; thus in principle, showing the skeletons of these constructions are same. Inductive groupoids are obtained from a biordered set by adjoining the groupoid along with restrictions/corestrictions; whereas a normal category is obtained from cross-connected partially ordered sets by adjoining right (left) translations. Observe that the role of structure mappings in the initial construction of Nambooripad \cite{kssthesis,ksssf1,stmp}, which was replaced by restrictions/corestrictions in \cite{mem}, was filled in by retractions/inclusions in the case of normal categories. 
	
	A diligent reader may have noticed that the isomorphisms in normal categories are obtained from the morphisms of the groupoids. Roughly speaking, a normal category is built by `pasting' morphisms of the groupoid into each of the partially ordered sets. The retractions and inclusions in the normal category come from the regular partially ordered sets.
	
	Observe that although normal categories may be seen as one-sided versions of inductive groupoids, it was relatively difficult to construct cross-connections from inductive groupoids; rather than the other way round. In some sense, the inductive groupoid sits inside the cross-connection structure as a cross-connected pair of ordered groupoids. But such a direct reverse identification does not seem evident, at least so far. This in turn suggests that the cross-connection representation encodes much more information regarding the semigroup compared to inductive groupoid.
	
	As mentioned in Section \ref{cxn}, Nambooripad's semigroup $T\mathcal{C}$ of all normal cones in a normal category $\mathcal{C}$ is an interesting object, in itself. The authors believe that it is some kind of a universal object in the category of regular semigroups `generated' by a normal category, much like the idempotent generated semigroup in \cite{mem}, and must be further explored.
	
	This may be particularly relevant in the discussion regarding the maximal subgroups of the free idempotent generated semigroups. This problem has its origins in \cite{mem,ksspastijn,eas} and it has regained recent interest following the unexpected result in \cite{fige} that subgroups of free idempotent generated semigroups need not be free. This has led to various new approaches and results in this area. See \cite{volkov,gould,gray,gray1,igd3} for instance. We believe that due to the close proximity and similarity of cross-connection theory with the biordered sets, it may be worth exploring the cross-connection structure of free (regular) idempotent generated semigroups; and its implications to the above problems. We also believe this may have a counter effect on the theory of cross-connections, especially in the context of generalisation of the theory to arbitrary semigroups \cite{newcross,newcross1}.
	
	Yet another species of regular semigroups whose cross-connection structure appears to be of interest consists of so-called bifree and trifree semigroups in e-va\-ri\-e\-ties. Recall that a class of regular semigroups is said to be an \emph{e-variety} if it is closed under taking direct products, regular subsemigroups and homomorphic images; this notion was introduced by Hall \cite{hall1} and, independently, by Ka\softd{}ourek and Szendrei \cite{kad-szen} for orthodox semigroups. The concept of an e-variety has proved to be very productive; in particular, it has allowed one to introduce rather natural versions of free objects for several important classes of regular semigroups that are known to admit no free objects in the standard sense. Concrete examples include, say, the bifree locally inverse semigroup (see \cite{auinger94,auinger95}) and the bifree E-solid semigroup (see \cite{szendrei96}); their cross-connection structure is certainly worth exploring. 
	
	Above all these, one place the cross-connection theory naturally finds at home should be the `ESN world'. This has been a flourishing area of research, where various non-regular generalisations of inverse semigroups like concordant semigroups, abundant semigroups, Ehresmann semigroups, ample semigroups, restriction semigroups, weakly $U$-regular semigroups etc have been described using generalisations of inductive groupoids like  inductive cancellative categories, Ehresmann categories, inductive categories, inductive constellations, weakly regular categories etc \cite{armstrong,lawsonordered0,lawsonenlarge,lawsonordered,gomes,gouldrestr,hollings2,gouldrestriction,gould2,hollings,wang,wangureg}. It must be mentioned here that the cross-connection construction has already been extended to concordant semigroups \cite{romeo,romeo1,conc}. Hence, it may be very natural to expect the other aforementioned semigroups also to have a rich cross-connection structure.
	
	On a personal note, all this justifies Nambooripad's unwavering confidence in his cross-connection theory, in spite of almost two decades of literal dormancy in the area! 
	
	\section*{Acknowledgements}
	We are very grateful to J. Meakin, University of Nebraska-Lincoln, for reading an initial draft of the article and helping us with several enlightening suggestions.
	

\begin{thebibliography}{99}
		
		\bibitem{armstrong}
		S.~Armstrong.
		\newblock Structure of concordant semigroups.
		\newblock {\em J. Algebra}, 118(1):205--260, 1988.
		
		\bibitem{auinger94}
		K.~Auinger.
		\newblock The bifree locally inverse semigroup on a set.
		\newblock {\em J. Algebra}, 166(3):630--650, 1994.
		
		\bibitem{auinger95}
		K.~Auinger.
		\newblock On the bifree locally inverse semigroup.
		\newblock {\em J. Algebra}, 178(2):581--613, 1995.
		
		\bibitem{var}
		P.~A. Azeef~Muhammed.
		\newblock Cross-connections and variants of the full transformation semigroup.
		\newblock {\em Acta Sci. Math. (Szeged)}, 2018.
		\newblock To appear, arXiv:1703.04139.
		
		\bibitem{tlx}
		P.~A. Azeef~Muhammed.
		\newblock Cross-connections of linear transformation semigroups.
		\newblock {\em Semigroup Forum}, 2018.
		\newblock To appear, arXiv:1701.06098.
		
		\bibitem{conc}
		P.~A. Azeef~Muhammed, K.~S.~S. Nambooripad, and P.~G. Romeo.
		\newblock Cross-connection structure of concordant semigroups.
		\newblock 2018.
		\newblock (In preparation).
		
		\bibitem{tx}
		P.~A. Azeef~Muhammed and A.~R. Rajan.
		\newblock Cross-connections of the singular transformation semigroup.
		\newblock {\em J. Algebra Appl.}, 17(3):1850047, 2018.
		
		\bibitem{indcxn2}
		P.~A. Azeef~Muhammed and M.~V. Volkov.
		\newblock Inductive groupoids and cross-connections.
		\newblock 2018.
		\newblock (In preparation).
		
		\bibitem{fige}
		M.~Brittenham, S.~W. Margolis, and J.~Meakin.
		\newblock Subgroups of free idempotent generated semigroups need not be free.
		\newblock {\em J. Algebra}, 321(10):3026--3042, 2009.
		
		\bibitem{clifbi}
		A.~H. Clifford.
		\newblock The partial groupoid of idempotents of a regular semigroup.
		\newblock {\em Semigroup Forum}, 10(1):262--268, 1975.
		
		\bibitem{clif}
		A.~H. Clifford and G.~B. Preston.
		\newblock {\em The Algebraic Theory of Semigroups, Volume 1}.
		\newblock Number~7 in Mathematical Surveys. American Mathematical Society,
		Providence, Rhode Island, 1961.
		
		\bibitem{gould}
		Y.~Dandan and V.~Gould.
		\newblock Free idempotent generated semigroups over bands and biordered sets
		with trivial products.
		\newblock {\em Internat. J. Algebra Comput.}, 26(03):473--507, 2016.
		
		\bibitem{igd3}
		I.~Dolinka and R.~Gray.
		\newblock Maximal subgroups of free idempotent generated semigroups over the
		full linear monoid.
		\newblock {\em Trans. Amer. Math. Soc.}, 366(1):419--455, 2014.
		
		\bibitem{eas}
		D.~Easdown.
		\newblock Biordered sets come from semigroups.
		\newblock {\em J. Algebra}, 96(2):581--591, 1985.
		
		\bibitem{volkov}
		D.~Easdown, M.~V. Sapir, and M.~V. Volkov.
		\newblock Periodic elements of the free idempotent generated semigroup on a
		biordered set.
		\newblock {\em Internat. J. Algebra Comput.}, 20(02):189--194, 2010.
		
		\bibitem{ehrs}
		C.~Ehresmann.
		\newblock Gattungen von lokalen {S}trukturen.
		\newblock {\em Jahresber. Deutsch. Math.-Verein.}, 60:49--77, 1957.
		
		\bibitem{ehrs1}
		C.~Ehresmann.
		\newblock Categories inductives et pseudogroupes.
		\newblock {\em Ann. Inst. Fourier (Grenoble)}, 10:307--336, 1960.
		
		\bibitem{gomes}
		G.~M. Gomes and V.~Gould.
		\newblock Fundamental {E}hresmann semigroups.
		\newblock {\em Semigroup Forum}, 63(1):11--33, 2001.
		
		\bibitem{gouldrestriction}
		V.~Gould.
		\newblock Restriction and {E}hresmann semigroups.
		\newblock In {\em Proceedings of the International Conference on Algebra 2010:
			Advances in Algebraic Structures}, page 265. World Scientific, 2011.
		
		\bibitem{gouldrestr}
		V.~Gould and C.~Hollings.
		\newblock Restriction semigroups and inductive constellations.
		\newblock {\em Comm. Algebra}, 38(1):261--287, 2009.
		
		\bibitem{gould2}
		V.~Gould and Y.~Wang.
		\newblock Beyond orthodox semigroups.
		\newblock {\em J. Algebra}, 368:209--230, 2012.
		
		\bibitem{gray1}
		R.~Gray and N.~Ruskuc.
		\newblock Maximal subgroups of free idempotent generated semigroups over the
		full transformation monoid.
		\newblock {\em Proc. London Math. Soc.}, 104(5):997--1018, 2012.
		
		\bibitem{gray}
		R.~Gray and N.~Ruskuc.
		\newblock On maximal subgroups of free idempotent generated semigroups.
		\newblock {\em Israel Journal of Mathematics}, 189(1):147--176, 2012.
		
		\bibitem{gril}
		P.~A. Grillet.
		\newblock Structure of regular semigroups: {A} representation.
		\newblock {\em Semigroup Forum}, 8:177--183, 1974.
		
		\bibitem{gril1}
		P.~A. Grillet.
		\newblock Structure of regular semigroups: {C}ross-connections.
		\newblock {\em Semigroup Forum}, 8:254--259, 1974.
		
		\bibitem{gril2}
		P.~A. Grillet.
		\newblock Structure of regular semigroups: {T}he reduced case.
		\newblock {\em Semigroup Forum}, 8:260--265, 1974.
		
		\bibitem{grillet}
		P.~A. Grillet.
		\newblock {\em Semigroups: An Introduction to the Structure Theory}.
		\newblock CRC Pure and Applied Mathematics. Taylor \& Francis, 1995.
		
		\bibitem{hall}
		T.~E. Hall.
		\newblock On regular semigroups.
		\newblock {\em J. Algebra}, 24(1):1--24, 1973.
		
		\bibitem{hall1}
		T.~E. Hall.
		\newblock Identities for existence varieties of regular semigroups.
		\newblock {\em Bull. Austral. Math. Soc.}, 40(1):59--77, 1989.
		
		\bibitem{hart}
		R.~E. Hartwig.
		\newblock How to partially order regular elements.
		\newblock {\em Math. Japon}, 25(1):1--13, 1980.
		
		\bibitem{higginscat}
		P.~J. Higgins.
		\newblock {\em Notes on categories and groupoids}.
		\newblock Van Nostrand Reinhold, 1971.
		
		\bibitem{hollings2}
		C.~Hollings.
		\newblock From right {PP}-monoids to restriction semigroups: a survey.
		\newblock {\em European Journal of Pure and Applied Mathematics}, 2(1):21--57,
		2009.
		
		\bibitem{hollings}
		C.~Hollings.
		\newblock The {E}hresmann{-}{S}chein{-}{N}ambooripad theorem and its
		successors.
		\newblock {\em European Journal of Pure and Applied Mathematics},
		5(4):414--450, 2012.
		
		\bibitem{kad-szen}
		J.~Ka\softd{}ourek and M.~Szendrei.
		\newblock A new approach in the theory of orthodox semigroups.
		\newblock {\em Semigroup Forum}, 40(3):257--296, 1990.
		
		\bibitem{knr}
		E.~Krishnan, K.~S.~S. Nambooripad, and A.~R. Rajan.
		\newblock {\em Theory of Regular Semigroups}.
		\newblock Centre for Mathematical Sciences, Thiruvananthapuram, 2002.
		\newblock (Preprint).
		
		\bibitem{lawsonordered0}
		M.~V. Lawson.
		\newblock Semigroups and ordered categories. {I}. the reduced case.
		\newblock {\em J. Algebra}, 141(2):422--462, 1991.
		
		\bibitem{lawsonenlarge}
		M.~V. Lawson.
		\newblock Enlargements of regular semigroups.
		\newblock {\em Proc. Edinb. Math. Soc. (2)}, 39(03):425--460, 1996.
		
		\bibitem{lawsonordered}
		M.~V. Lawson.
		\newblock Ordered groupoids and left cancellative categories.
		\newblock {\em Semigroup Forum}, 68(3):458--476, 2004.
		
		\bibitem{lawson}
		M.V. Lawson.
		\newblock {\em Inverse Semigroups: The Theory of Partial Symmetries}.
		\newblock World Scientific Pub. Co. Inc., 1998.
		
		\bibitem{sunny}
		S.~Lukose and A.~R. Rajan.
		\newblock Ring of normal cones.
		\newblock {\em Indian J. Pure Appl. Math.}, 41(5):663--681, 2010.
		
		\bibitem{mac}
		S.~MacLane.
		\newblock {\em Categories for the Working Mathematician}, volume~5 of {\em
			Graduate Texts in Mathematics}.
		\newblock Springer-Verlag, New York, 1971.
		
		\bibitem{stmp0}
		J.~C. Meakin.
		\newblock On the structure of inverse semigroups.
		\newblock {\em Semigroup forum}, 12:6--14, 1976.
		
		\bibitem{stmp}
		J.~C. Meakin.
		\newblock The structure mappings on a regular semigroup.
		\newblock {\em Proc. Edinb. Math. Soc. (2)}, 21(2):135--142, 1978.
		
		\bibitem{stmp1}
		J.~C. Meakin.
		\newblock Structure mappings, coextensions and regular four-spiral semigroups.
		\newblock {\em Trans. Amer. Math. Soc.}, 255:111--134, 1979.
		
		\bibitem{jmar}
		J.~C. Meakin and A.~R. Rajan.
		\newblock Tribute to {K. S. S. Nambooripad}.
		\newblock {\em Semigroup Forum}, 91(2):299--304, 2015.
		
		\bibitem{mitsch}
		H.~Mitsch.
		\newblock A natural partial order for semigroups.
		\newblock {\em Proc. Amer. Math. Soc.}, 97(3):384--388, 1986.
		
		\bibitem{munn}
		W.~D. Munn.
		\newblock Fundamental inverse semigroups.
		\newblock {\em Q. J. Math.}, 21(2):157--170, 1970.
		
		\bibitem{kssthesis}
		K.~S.~S. Nambooripad.
		\newblock {\em Structure of Regular Semigroups}.
		\newblock PhD thesis, University of Kerala, India, 1973.
		
		\bibitem{ksssf}
		K.~S.~S. Nambooripad.
		\newblock Structure of regular semigroups. {I}. {F}undamental regular
		semigroups.
		\newblock {\em Semigroup Forum}, 9:354--363, 1975.
		
		\bibitem{ksssf1}
		K.~S.~S. Nambooripad.
		\newblock Structure of regular semigroups. {II}. {T}he general case.
		\newblock {\em Semigroup Forum}, 9:364--371, 1975.
		
		\bibitem{bicxn}
		K.~S.~S. Nambooripad.
		\newblock Relations between cross-connections and biordered sets.
		\newblock {\em Semigroup Forum}, 16:67--82, 1978.
		
		\bibitem{mem}
		K.~S.~S. Nambooripad.
		\newblock {\em Structure of Regular Semigroups. {I}}.
		\newblock Number 224 in Mem. Amer. Math. Soc. American Mathematical Society,
		Providence, Rhode Island, 1979.
		
		\bibitem{kssnpo}
		K.~S.~S. Nambooripad.
		\newblock The natural partial order on a regular semigroup.
		\newblock {\em Proc. Edinb. Math. Soc. (2)}, 23(03):249--260, 1980.
		
		\bibitem{cross0}
		K.~S.~S. Nambooripad.
		\newblock {\em Structure of Regular Semigroups. {II}. Cross-connections}.
		\newblock Publication No. 15. Centre for Mathematical Sciences,
		Thiruvananthapuram, 1989.
		
		\bibitem{cross}
		K.~S.~S. Nambooripad.
		\newblock {\em Theory of Cross-connections}.
		\newblock Publication No. 28. Centre for Mathematical Sciences,
		Thiruvananthapuram, 1994.
		
		\bibitem{newcross}
		K.~S.~S. Nambooripad.
		\newblock Cross-connections.
		\newblock In {\em Proceedings of the International Symposium on Semigroups and
			Applications}, pages 1--25, Thiruvananthapuram, 2007. University of Kerala.
		
		\bibitem{newcross1}
		K.~S.~S. Nambooripad.
		\newblock Cross-connections, 2014.
		\newblock http://www.sayahna.org/crs/.
		
		\bibitem{ksspastijn}
		K.~S.~S. Nambooripad and F.~Pastijn.
		\newblock Subgroups of free idempotent generated regular semigroups.
		\newblock {\em Semigroup Forum}, 21(1):1--7, 1980.
		
		\bibitem{rajancat}
		A.~R. Rajan.
		\newblock Certain categories derived from normal categories.
		\newblock In P.~G. Romeo, J.~C. Meakin, and A.~R. Rajan, editors, {\em
			Semigroups, Algebras and Operator Theory: Kochi, India, February 2014}, pages
		57--66, New Delhi, 2015. Springer India.
		
		\bibitem{rajancat1}
		A.~R. Rajan.
		\newblock Structure theory of regular semigroups using categories.
		\newblock In S.~T. Rizvi, A.~Ali, and V.~D. Filippis, editors, {\em Algebra and
			its Applications: ICAA, Aligarh, India, December 2014}, pages 259--264,
		Singapore, 2016. Springer Singapore.
		
		\bibitem{rajancat2}
		A.~R. Rajan.
		\newblock Inductive groupoids and normal categories of regular semigroups.
		\newblock 2017.
		\newblock (Preprint).
		
		\bibitem{romeo}
		P.~G. Romeo.
		\newblock {\em Cross connections of Concordant Semigroups}.
		\newblock PhD thesis, University of Kerala, India, 1993.
		
		\bibitem{romeo1}
		P.~G. Romeo.
		\newblock Concordant semigroups and balanced categories.
		\newblock {\em Southeast Asian Bull. Math.}, 31(5):949--961, 2007.
		
		\bibitem{schein}
		B.~M. Schein.
		\newblock On the theory of generalised groups and generalised heaps.
		\newblock {\em The Theory of Semigroups and its Applications, Saratov State
			University, Russia}, 1:286--324, 1965.
		\newblock (Russian).
		
		\bibitem{schein1}
		B.~M. Schein.
		\newblock On the theory of inverse semigroups and generalised groups.
		\newblock {\em Amer. Math. Soc. Transl. Ser. 2}, 113:89--122, 1979.
		\newblock (English translation).
		
		\bibitem{szendrei96}
		M.~Szendrei.
		\newblock The bifree regular {E}-solid semigroups.
		\newblock {\em Semigroup Forum}, 52(1):61--82, 1996.
		
		\bibitem{wang}
		S.~Wang.
		\newblock An {E}hresmann{-}{S}chein{-}{N}ambooripad{-}type theorem for a class
		of {P}{-}restriction semigroups.
		\newblock {\em Bull. Malays. Math. Sci. Soc.}, pages 1--34, 2017.
		
		\bibitem{wangureg}
		Y.~Wang.
		\newblock Beyond regular semigroups.
		\newblock {\em Semigroup Forum}, 92(2):414--448, 2016.
		
	\end{thebibliography}
	\end{document}